\documentclass[12pt,a4paper]{amsart}
\usepackage[latin1]{inputenc}
\usepackage{amsxtra,amssymb,amsthm,amsmath,amscd,mathrsfs}
\usepackage[all]{xy}

\usepackage[T1]{fontenc}
\usepackage[french]{babel}

\newcommand*{\C}{\mathbb{C}}

\newcommand*{\HH}{\mathbb{H}}

\newcommand*{\N}{\mathbb{N}}

\newcommand*{\Q}{\mathbb{Q}}
\newcommand*{\R}{\mathbb{R}}

\newcommand*{\Z}{\mathbb{Z}}
\def\le{\leqslant}
\def\leq{\leqslant}
\def\ge{\geqslant}
\def\geq{\geqslant}

\newcommand*{\sumh}{\mathop{\sideset{}{^\mathrm{h}}\sum}}
\newcommand*{\suma}{\mathop{\sideset{}{^\mathrm{*}}\sum}}

\newcommand*{\dd}{\,\mathrm{d}}
\newcommand*{\re}{\Re e\,}
\newcommand*{\im}{\Im m\,}
\theoremstyle{plain}
\newtheorem{thm}{Th\'eor\`eme}
\newtheorem{lemme}[thm]{Lemme}
\newtheorem{corollaire}[thm]{Corollaire}
\newtheorem{proposition}[thm]{Proposition}

\theoremstyle{definition}

\begin{document}

\title[Non annulation des fonctions $L$ automorphes au point central]
{Non annulation des fonctions $L$ automorphes au point central}

\author{D. Rouymi}
\address{Institut Elie Cartan Nancy
\\
CNRS, Nancy-Universit\'e, INRIA
\\
Boulevard des Aiguillettes, B.P. 239
\\
54506 Van\-d\oe uvre-l\`es-Nancy
\\
France}
\email{rouymi@yahoo.fr}

\date{\today}

\begin{abstract}
Les travaux sur les formes modulaires sont nombreux et divers; 
concernant leurs annulations Michel, Kowalski \& Vanderkam montrent (en autre) 
qu'il existe une proportion positive des formes qui ne s'annulent pas au point critique. 
Ce r\'esultat fut montr\'e par ces derniers pour des formes de niveau premier; 
d'autre part Iwaniec, Luo \& Sarnak montrent que ceci se g\'en\'eralise aux formes dont le niveau est sans facteur carr\'e.
Dans le but de comprendre l'influence de l'arithm\'etique du niveau sur les z\'eros de ces formes, 
cet article pr\'esente une \'etude de la g\'en\'eralisation aux formes primitives 
dont le niveau est la puissance d'un nombre premier.
\end{abstract}
\maketitle

\tableofcontents

\section{Introduction}

L'objet de cet article est d'\'etudier la non-annulation des fonctions $L$ de formes primitives 
de poids $k$ et de niveau ${\mathfrak{p}}^\nu$, 
o\`u $\mathfrak{p}$ est un nombre premier fix\'e et $\nu \to \infty$.

Nous commen\c cons par un rappel rapide de quelques notions.
On appelle forme parabolique  de poids $ k\geq{2} $ pair et de niveau $q$, toute fonction $f$ holomorphe sur le demi plan de Poincar\'e $\HH:=\{z\in \C :  \im {z} >0\}$
telle que 
\begin{equation}\label{relationmodulaire}
f\left(\frac{az+b}{cz+d}\right)=(cz+d)^kf(z) 
\end{equation}
pour tout \'el\'ement  
$$\begin{pmatrix}
a & b
\\
c & d
\end{pmatrix}
\in 
\Gamma_0(q) := \left\{
\begin{pmatrix}
a & b
\\
c & d
\end{pmatrix}
\in SL_2(\Z) : q\mid c\right\}.
$$
et que la fonction $z\mapsto (\im z)^{k/2}f(z)$ est born\'ee sur $\HH$.
On d\'esigne par $S_{k}(q)$ l'espace des formes paraboliques de poids $k$ et de niveau $q$, 
que l'on munit du produit scalaire  
$$
\langle f, g\rangle_q
: = \int_{F} f(z)\overline{g}(z) y^ k\frac{{\rm d}x {\rm d}y}{y^2},
$$
o\`u $F$ d\'esigne un domaine fondamental par l'action homographique de $\Gamma_0(q)$  
sur $\Q\cup{\left\{\infty\right\}}$. 
Pour chaque $f\in S_{k}(m) $ avec $m\mid q$, $m<q $ et $ l\mid (q/m)$ alors
$ z\mapsto f(lz)$ est une forme parabolique de $\Gamma_0(q)$.
De telles formes s'appellent des formes anciennes de niveau $q$.  
L'orthogonal de l'espace engendr\'e par ces formes est l'espace des formes nouvelles,
not\'ee par $S_{k}^{*}(q)$.
D\'esignons par ${\rm H}_{k}^{*}(q)$ la base orthogonale de $S_{k}^{*}(q)$ consitu\'ee des formes primitives. Ses \'el\'ements sont des fonctions propres des op\'erateurs de Hecke (cf. \cite[Paragraphes 2.7 et 3.3]{Iw 97}).

Toute forme $f\in S_k(q)$ a un d\'evelopeent de Fourier en $\infty$ : 
\begin{equation}\label{1}
f(z)=\sum_{n\geq1} a_f(n) e(nz),
\end{equation}
o\`u $e(t):=e^{2\pi it}$.
On pose
\begin{equation}\label{2}
\lambda_f(n) := a_f(n) n^{-(k-1)/2}.
\end{equation}
Quand $f\in {\rm H}_k^*(q)$, on a
\begin{align}
\lambda_f(1)
& =1, 
\label{5a}
\\\noalign{\vskip 1mm}
\lambda_f(n)
& \in \R,
\label{5b}
\\\noalign{\vskip -2mm}
\lambda_f(m)\lambda_f(n)
& = \sum_{\substack{d\mid (m, n)\\ (d, q)=1}} \lambda_f\bigg(\frac{mn}{d^2}\bigg)
\label{5}
\end{align}
pour tous les entiers $m$ et $n$ $\geq 1$.
En particulier, on utilisera dans la suite, que si $f\in {\rm H}_{k}^{*}(m')$ avec $m'$ entier $\geq 2$ tel que $m'\mid q$ et si on a $r$ et $r'$ entiers $\geq 1$ tels que
$r\mid q^\infty$ ou $r'\mid q^\infty$ alors :
\begin{equation}\label{6}
{\lambda_{f}(rr')=\lambda_f(r)\lambda_{f}(r')}
\end {equation}
de plus les travaux de Deligne montrent que si $f\in {\rm H}_{k}^{*}(m')$ alors
\begin{equation}\label{7}
\lambda_f(p)^2
= \begin{cases}
0    & \text{si $p^2\,|\,m'$},
\\
1/p & \text{si $p\;\|\,m'$}.
\end{cases}
\end{equation}

La fonction $L$ automorphe associ\'ee \`a $f\in {\rm H}_k^*(q)$ est d\'efinie par
\begin{equation}\label{3}
L(s, f):=\sum_{n\geq 1} \lambda_f(n) n^{-s}
\qquad(\re s>1). 
\end{equation}
D\'efinissons la fonction $L$ compl\`ete
\begin{equation}\label{69}
\Lambda(s, f) := \hat{q}^{s} \Gamma\bigg(s+{\frac{k-1}{2}}\bigg) L(s, f),
\end{equation}
o\`u
\begin{equation}\label{defhatq}
\hat{q}:=\sqrt{q}/(2\pi).
\end{equation}
Alors cette fonction peut \^etre prolong\'ee analytiquement sur $\C$ et v\'erifie l'\'equation fonctionnelle :
\begin{equation}\label{EF}
{\Lambda(s,f)=\varepsilon_f\Lambda(1-s,f)}
\qquad
(s\in \C)
\end{equation}
o\`u \ $\varepsilon_f=1$ ou $-1$.

Les valeurs sp\'eciales de $L(s, f)$ (par exemple, valeurs centrales et valeurs au bord de la bande critique)
contiennent des informations int\'eressantes.
En particulier,
la non annulation de $L(s, f)$ au point central $s=\frac{1}{2}$ est 
une des questions centrales en th\'eorie des fonctions $L$ automorphes
et a beaucoup d'applications dans divers probl\`emes.
D'apr\`es Gross et Zagier \cite{GZ86}, nous savons que
\begin{equation}\label{positivity}
L(\textstyle\frac{1}{2}, f)\geq 0.
\end{equation}
Un lien surprenant avec le z\'ero de Landau-Siegel a \'et\'e decouvert par Iwaniec \& Sarnak \cite{IS00}.
En d\'esignant par $\varphi(q)$ la fonction d'Euler et
${\rm H}^{+}_{k}(q)$ (resp. ${\rm H}^{-}_{k}(q)$) est 
l'ensemble de $f\in {\rm H}^{*}_{k}(q)$ avec $\varepsilon_f=1$ (resp. $\varepsilon_f=-1$),
leur r\'esultat s'\'enonce comme suit : si $q$ est sans facteur carr\'e assez grand 
tel que $\varphi(q)\gg q$ alors
$$
\frac{1}{\left|{\rm H}^{+}_{k}(q)\right|}
\sum_{\substack{f\in {\rm H}^{+}_{k}(q)\\ L(\frac{1}{2}, f)\geq (\log q)^{-2}}}1
\geq \frac{1}{2}$$
et si l'on peut remplacer $\frac{1}{2}$ par une constante $c>\frac{1}{2}$, alors il n'existe pas le z\'ero de Landau-Siegel pour les fonctions $L$ de Dirichlet. 

Le premier r\'esultat concernant la non annulation de $L(\frac{1}{2}, f)$ a \'et\'e obtenu par Duke \cite{D95}.
Il a d\'emontr\'e que si $q$ est un nombre premier avec $q\geq 11$ et $q\neq 13$ alors il existe une constante absolue $C>0$ telle que:
\begin{equation}\label{Duke}
\frac{1}{\left|{\rm H}^{*}_2(q)\right|}
\sum_{\substack{f\in {\rm H}^{*}_2(q)\\L(\frac{1}{2}, f)\neq 0}}1
\geq \frac{C}{(\log q)^2}\cdot
\end{equation}
Cette minoration est obtenue avec la formule de Petersson \cite[Lemma 1, p.167]{D95}.
Par la suite, Kowalski \& Michel \cite{KM00} obtiennent une proportion positive de non-annulation, 
\`a savoir: si $q$ est un nombre premier assez grand alors
\begin{equation}\label{KowalskiMichel}
\frac{1}{\left|{\rm H}^{*}_{2}(q)\right|}
\sum_{\substack{f\in {\rm H}^{*}_{2}(q)\\ L(\frac{1}{2}, f)\neq 0}}1\geq \frac{19}{54}\cdot
\end{equation}
Ce r\'esultat est d\'emontr\'e avec la formule de trace de Petersson et le calcul des moments $1$ et $2$ des fonctions $L$ mollifi\'e
\footnote{C'est cette technique de mollification 
qui permet de supprimer le facteur $\log$ dans le r\'esultat de Duke}.
En m\^eme temps (ind\'ependemment), 
Vanderkam \cite{V99} applique la formule de trace de Selberg \cite[Propopsition 4]{Ser} aux deux premiers moments de la fonction $L$ pour obtenir (\ref{KowalskiMichel}) 
avec une constante l\'eg\`erement moins bonne
$\frac{1}{48}$ \`a la place de $\frac{19}{54}$.
Notons que ce r\'esultat est obtenu \'egalement (sous une forme diff\'erente) 
par Kowalski, Michel \& Vanderkam \cite{KMV00}.
Enfin quand $q$ est sans facteur carr\'e, 
Iwaniec, Luo \& Sarnak \cite{ILS} montrent 
$$\liminf_{q\to \infty}
\frac{1}{\left|{\rm H}^{+}_{k}(q)\right|}
\sum_{\substack{f\in {\rm H}^{+}_{k}(q)\\ L(\frac{1}{2}, f)\neq 0}}1
\geq \frac{9}{16}\cdot$$
Dans le m\^eme article ils \'etablissent une fomule de trace sp\'ecifique au cas 
o\`u $q$ est sans facteur carr\'e.

D'autre part,
l'\'etude sur les valeurs extr\^emes de $L(1, {\rm sym}^mf)$
(fonction $L$ de la $m$-\`eme puissance sym\'etrique associ\'ee \`a $f$)
a re\c cu beaucoup d'attention (voir \cite{CM04}, \cite{RW05} et \cite{LW08}).
En particulier, les r\'esultats de Royer \& Wu \cite{RW05} montrent que
les valeurs extr\^emes de $L(1, {\rm sym}^mf)$ d\'ependent, d'une mani\`ere suprenante,
des propri\'et\'es arithm\'etiques du niveau.
Donc il est naturel d'\'etudier l'influence de l'arithm\'etique du niveau 
sur le probl\`eme de non annulation de $L(\frac{1}{2}, f)$.
Dans cet article, nous proposons de minorer le quotient
$$
\frac{1}{\left|{\rm H}^{*}_{k}(q)\right|}
\sum_{\substack{f\in {\rm H}^*_k(q)\\ L(\frac{1}{2}, f)\neq 0}}1
$$
pour des entiers $q$ de la forme $p^{\nu}$, o\`u $p$ est un nombre premier et $\nu\ge 1$ est un entier.
Le choix de cette forme de niveau a deux raisons :
premi\`erement en prenant $\nu=1$, nous retrouvons le cas classique qui a \'et\'e \'etudi\'e
par Duke \cite{D95}, Kowalski \& Michel \cite{KM00} et Vanderkam \cite{V99}, mentionn\'e ci-dessus ; deuxi\`emement, en fixant $p$ et faisant $\nu\to\infty$,
on obtient un cas de niveau vraiment friable (i.e. il n'y a que les facteurs premiers petits).
Ce cas extr\^eme arithm\'etiquement contraire au cas de niveau premier 
nous aidera \`a comprendre l'influence de l'arithm\'etique du niveau 
sur le probl\`eme de non annulation.

On notera 
\begin{equation}\label{4}
\omega_q(f) := \frac{\Gamma(k-1)}{(4\pi)^{k-1}} \left\langle f, f\right\rangle_q.
\end{equation}
Pour une partie $A$ de $S_{k}(q)$ on d\'efinit la somme  harmonique :
$$
\sumh_{f\in A}\alpha_f := \sum_{f\in A} \alpha_f \omega_q(f).
$$

Dans cet article, nous montrerons le r\'esultat suivant.

\begin{thm}\label{thm_1}
Soient $k\ge 2$ un entier pair et $\mathfrak{p}$ un nombre premier.
Alors il existe une constante $\nu_0(k, \mathfrak{p})$ telle que
pour $\nu\ge \nu_0(k, \mathfrak{p})$ et $q=\mathfrak{p}^\nu$ on a
$$
\sumh_{\substack{f\in{{\rm H}^{*}_{k}(q)}\\ L(\frac{1}{2}, f)\neq 0}} 1
\gg \frac{1}{(\log q)^3},
$$
o\`u la constante impliqu\'ee ne d\'epend que de $k$ et $\mathfrak{p}$.
\end{thm}

Pour ce faire, dans un premier temps, 
on \'etablira une formule de trace sur les formes primitives de niveau 
$\mathfrak{p}^\nu$ qui prendra la forme suivante :
\begin{equation}\label{9}
\begin{aligned}
\Delta^*_q(m,n)
& :=\sum_{f\in {\rm H}_{k}^{*}(q)} \omega_q(f) \lambda_f(m) \lambda_f(n)
\\
& \,= \frac{\varphi(q)}{q} 
\delta_{m, n} +R(m,n,k,q),
\end{aligned}
\end{equation}
o\`u $\varphi(q)$ est la fonction d'Euler et $\delta(m, n)$ est le symbole de Kronecker
(voir le Th\'eor\`eme \ref{thm_2} ci-dessous).
Cette formule de trace sera \'etablie de la mani\`ere suivante :
\begin{itemize}
\item
Nous commen\c cons par une formule de trace sous la forme
\begin{equation}\label{8}
\Delta_q(m, n) := \sum_{f\in B_k(q)} \omega_q(f) \lambda_f(m) \lambda_f(n),
\end{equation}
o\`u $B_k(q)$ est une base orthogonale quelconque de $S_{k}(q)$.
Il est \`a noter que cette d\'efinition est ind\'ependante du choix de la base orthogonale 
puisque $\Delta_q(m, n)$ est le coefficient de Fourier d'une s\'erie de Poincar\'e \cite[Lemma 3.3]{DI82}.
\`A l'aide d'une d\'ecomposition permettant de passer des formes paraboliques aux formes primitives de niveaux inf\'erieurs, on peut exprimer $\Delta_q(m, n)$ en fonction des nombres $\Delta^{*}_{q'}(m,n)$, 
o\`u $q'\mid q$, tout en rendant n\'egligeable la contribution des formes de niveau $1$. 
Puis par inversion de M\"obius on pourra exprimer $\Delta^*_q(m, n)$ 
en fonction des nombres $\Delta_{q'}(m, n)$.
\item
Apr\`es avoir \'etabli une formule de trace dans $S_{k}(q)$ (de type \cite{ILS} \'egalit\'e $(2.12)$) provenant de l'expression de $\Delta_q(m, n) $ comme des sommes de sommes de Kloosterman, on en d\'eduit alors une formule de trace dans ${\rm H}^{*}_{k}(q)$.
\end{itemize}

\smallskip

Dans un second temps, on calculera au quatri\`eme et cinqui\`eme paragraphe, 
le deuxi\`eme et le troisi\`eme moment au point critique et ce \`a l'aide de la formule trace, pour obtenir : 
\begin{align*}
M_2
& = \bigg(\frac{\varphi(q)}{q}\bigg)^2 \log{q}+O_{k, \mathfrak{p}}(1),
\\
M_3
& = \frac{1}{6} \bigg(\frac{\varphi(q)}{q}\bigg)^4 (\log{q})^3+O_{k, \mathfrak{p}}\big((\log q)^2\big),
\end{align*}
o\`u 
\begin{equation}\label{defMr}
M_r := \sumh_{f\in {\rm H}^{*}_{k}(q)} L({\textstyle\frac{1}{2}}, f)^r.
\end{equation}
Enfin une simple application de l'in\'egalit\'e de H\"older donne le Th\'eor\`eme \ref{thm_1}.

\smallskip

\noindent{\bf Notations.}
Dans ce texte, $\tau(n)$ (resp. $\omega(n)$) est le nombre des diviseurs de $n$ (resp. le nombre de facteurs premiers distincts) et  $\varphi(n)$ la fonction indicatrice d'Euler.

\smallskip

\noindent{\bf Remerciements.}
L'auteur tient \`a remercier ses directeurs de th\`ese Jie Wu (Nancy) et Emmanuel Royer (Clermont-Ferrand) pour toute leur patience et leurs encouragements r\'eguliers durant l'\'elaboration de ce travail.

\section{Formule de trace harmonique au niveau $\mathfrak{p}^\nu$ avec $\nu\geq 1$}

Le but de ce paragraphe est d'\'etablir une formule de trace au niveau $\mathfrak{p}^\nu$ avec $\nu\geq 1$.
Notre r\'esultat peut \^etre consid\'er\'e comme compl\'ementaire au Corollaire 2.10 de
Luo, Iwaniec \& Sarnak \cite{ILS}.

\subsection{Enonc\'e du r\'esultat}

\begin{thm}\label{thm_2}
Soient $k\ge 2$ un entier pair,
$\mathfrak{p}$ un nombre premier et $q=\mathfrak{p}^{\nu}$ avec $\nu\geq 1$.
Alors pour tous entiers $m\ge 1$ et $n\ge 1$, on a
\begin{equation}\label{10}
\Delta^*_q(m, n) 
= \begin{cases}
\phi(\nu, \mathfrak{p})\delta_{m, n} + O({\mathscr R})
& \text{si $\mathfrak{p}\nmid mn$ et $\nu\ge 1$},
\\\noalign{\vskip 1mm}
0 & \text{si $\mathfrak{p}\mid mn$ et $\nu\ge 2$},
\end{cases}
\end{equation}
o\`u $\delta_{m, n}$ est le symbole de Kronecker,
\begin{equation}\label{defphinup}
\phi(\nu, \mathfrak{p})
:=\begin{cases}
1                                                                & \text{si $\nu=1$}
\\\noalign{\vskip 1mm}
1-(\mathfrak{p}-\mathfrak{p}^{-1})^{-1} & \text{si $\nu=2$}
\\\noalign{\vskip 1mm}
1-\mathfrak{p}^{-1}                                  & \text{si $\nu\ge 3$}
\end{cases}
\end{equation}
et
\begin{equation}\label{defR}
{\mathscr R}
:= \frac{\sqrt{mn\mathfrak{p}} \{\log(2(m,n))\}^2}{k^{4/3}q^{3/2}}
+\frac{\tau(m)\tau(n)}{q}.
\end{equation}
La constante impliqu\'ee est absolue.
Le deuxi\`eme terme d'erreur $\tau(m)\tau(n)/q$ n'existe que 
s'il y a des formes de poids $k$ et de niveau 1.
\end{thm}

\begin{corollaire}\label{cor_3}
Soient $k\ge 2$ un entier pair,
$\mathfrak{p}$ un nombre premier et $q=\mathfrak{p}^{\nu}$ avec $\nu\geq 3$.
Alors pour tous entiers $m\ge 1$ et $n\ge 1$, on a
$$
\Delta^{*}_q(m,n) = \begin{cases}
\displaystyle\frac{\varphi(q)}{q} \delta_{m, n}
+O_{k, \mathfrak{p}}\bigg(\frac{\sqrt{mn} \{\log(2(m,n))\}^2}{q^{3/2}}
+\frac{\tau(m)\tau(n)}{q}\bigg)
& \text{si $\mathfrak{p}\nmid mn$},
\\\noalign{\vskip 1mm}
0 & \text{si $\mathfrak{p}\mid mn$},
\end{cases}
$$
o\`u la constante impliqu\'ee ne d\'epend que de $k$ et $\mathfrak{p}$.
\end{corollaire}

{\sl Remarque 1}.
Si on applique le Th\'eor\`eme~\ref{thm_2} \`a $\nu=1$, 
on retrouve si $ k\in \mathcal{K}=\left\{2 , 4 , 6 , 8 , 10 , 14 \right\}$, 
la formule de trace (\ref{10}) en tenant compte du fait qu'alors 
$S_{k}(\mathfrak{p})=S^{*}_{k}(\mathfrak{p})$ \cite[Chap. 7]{Ser} et que donc 
$\Delta^{*}_{\mathfrak{p}}=\Delta_{\mathfrak{p}}$ 
et du fait qu'alors le deuxi\`eme terme de droite dans (\ref{defR}) est nul. 

\subsection{Lemmes auxiliaires}
Commen\c cons par \'etablir une formule de trace vraie dans tout l'espace des formes paraboliques de niveau $\mathfrak{p}^{\nu}$ avec $\nu\geq 1$ et de poids $k$.

\begin{lemme}\label{lem_4}
Soient $k\ge 2$ un entier pair, $\mathfrak{p}$ un nombre premier,
$m\ge 1, n\ge 1$ et $q=\mathfrak{p}^{\nu}$ avec $\nu\ge 0$. 
Alors
\begin{equation}\label{14}
\Delta_q(m, n)=\delta_{m, n} 
+O\bigg(
\frac{\sqrt{mn(m,n,q)} \{\log(2(m,n))\}^2}{k^{4/3}q^{3/2}}
\bigg)
\end{equation}
o\`u la constante impliqu\'ee est absolue.
\end{lemme}

\begin{proof}
Selon \cite[Page 248-9]{DI82}, on a
$$
\Delta_q(m, n)=\delta_{m, n} 
+2\pi     i^k\sum_{c\equiv0({\rm mod}\,q)}\frac{S(m,n;c)}{c}J_{k-1}\left(\frac{4\pi\sqrt{mn}}{c}\right),
$$
o\`u $S(m,n;c)$ est la somme de Kloosterman d\'efinie par
\[S(m,n;c)=\sum_{dd'\equiv1({\rm mod}\,c)}\exp{\left(2\pi i\left(\frac{dm+d'n}{c}\right)\right)}\]
et $J_{k-1}$ est la fonction de Bessel de premi\`ere esp\`ece.
En utilisant les majorations classiques (\cite[Pages 60-1]{Iw 97} et \cite[Page 245]{DI82}) :
$$
|S(m,n;c)|\leq 2^{\omega(c)} (m,n,c)^{1/2} c^{1/2},
\qquad
J_{k-1}(x)\ll k^{-4/3} x,
$$
on peut d\'eduire
\begin{align*}
\Delta_q(m,n)
& = \delta_{m, n} +O\bigg(\frac{\sqrt{mn}}{k^{4/3}q^{3/2}}
\sum_{r\geq 1} \frac{2^{\omega(qr)}}{r^{3/2}}(m,n,qr)^{1/2}\bigg).
\end{align*}
Puisque $\omega(qr)\leq \omega(r)+1$
et $(m,n,qr)\mid(m,n,q)(m,n,r)$, il suit, en posant $d=(m,n,r)$ et $r=d\ell$,
\begin{align*}
\Delta_q(m, n)
& =\delta_{m, n} 
+O\bigg(\frac{\sqrt{mn(m,n,q)}}{k^{4/3}q^{3/2}}\sum_{r\geq 1}\frac{2^{\omega(r)}(m,n,r)^{1/2}}{r^{3/2}}\bigg)
\\
& =\delta_{m, n} 
+O\bigg(\frac{\sqrt{mn(m,n,q)}}{k^{4/3}q^{3/2}}
\sum_{d\mid (m, n)} \frac{2^{\omega(d)}}{d} \sum_{\ell\geq 1}\frac{2^{\omega(\ell)}}{\ell^{3/2}}\bigg)
\\
& =\delta_{m, n} 
+O\bigg(\frac{\sqrt{mn(m,n,q)}}{k^{4/3}q^{3/2}} \log^2(2(m, n))\bigg),
\end{align*}
o\`u l'on a d\'ej\`a utilis\'e les estimations classiques 
(voir (\ref{estimation2}) du Lemme \ref{lem_estimation} ci-dessous)
\begin{align*}
\sum_{d\mid (m, n)} \frac{2^{\omega(d)}}{d}
& \le \sum_{d\le (m, n)} \frac{\tau(d)}{d}
= \int_{1-}^{(m, n)} \frac{1}{t} \dd \sum_{d\le t} \tau(d)
\\
& = \int_{1-}^{(m, n)} \frac{1}{t} \dd O(t\log t)
\ll \log^2(2(m,n)).
\end{align*}
Cela ach\`eve la d\'emonstration.
\end{proof}

Dans le but d'exprimer $\Delta_q(m, n)$ en fonction de $\Delta^*_q(m, n)$, 
on utilise la d\'ecomposition orthogonale :
\begin{equation}\label{15}
S_{k}(q)=\bigoplus_{\ell m'=q} \bigoplus_{f\in {\rm H}_{k}^{*}(m')} S_k(\ell, f)
\end{equation} 
o\`u (si $f\in{{\rm H}_{k}^{*}(m')}$), $S_k(\ell, f) $ est l'espace engendr\'e par les formes :
\begin{equation}\label{16}
f_{|d}(z):=d^{k/2} f(dz)
\end{equation}
o\`u $d$ d\'esigne un diviseur de $\ell$.

\'Etant donn\'e la d\'efinition intrins\`eque (\ref{8}) de $\Delta_q$, il sera n\'ecessaire de d\'eterminer une base orthogonale de $S_k(\ell, f)$ pour tout diviseur $\ell$ de $q$. Pour cela, on introduit des fonctions de la forme:
\[f_d =\sum_{c\mid \ell}x_{d}(c,f)f_{|c}\]
o\`u $q=\ell m'$, $d$ est un diviseur de $\ell$, $f\in {\rm H}_{k}^{*}(m') $. 

Si $ m'>1$, les coefficients $x_{d}(c,f)$ sont d\'efinis de la fa\c con suivante:
\begin{equation}\label{17}x_{d}(c,f):=\begin{cases}
\displaystyle\frac{\mu{(r)}\lambda_f(r)}{\sqrt{r\rho_{f,m'}(d)}} & \text {si $ d=rc$ ,}
\\\noalign{\vskip 1mm}
0 & \text{sinon}
\end{cases}
\end{equation}
o\`u
\begin{equation}\label{18}
\rho_{f, m'}(d) := \sum_{n \mid d} \mu(n)\lambda_f(n)^2 n^{-1}.
\end{equation}

Si $m'=1$, on d\'efinit
\begin{equation}\label{21}
f_{\mathfrak{p}^r}:=\begin{cases}
f 
& \text{si $r=0$},
\\\noalign{\vskip 2mm}
\displaystyle 
\frac{1}{\sqrt{\sigma_f}} \left(f_{|\mathfrak{p}}-\frac{P_1(\lambda_f(\mathfrak{p}))}{\sqrt{\mathfrak{p}}}f_{|1}\right)
& \text{si $r=1$},
\\
\displaystyle 
\frac{1}{\sqrt{(1-\mathfrak{p}^{-2})\sigma_f}} \left(f_{|\mathfrak{p}^r}
-\nu'(\mathfrak{p})\frac{P_1(\lambda_f(\mathfrak{p}))}{\sqrt{\mathfrak{p}}} f_{|\mathfrak{p}^{r-1}}
+\frac{1}{\mathfrak{p}} f_{|\mathfrak{p}^{r-2}}\right)
& \text{si $r\ge 2$},
\end{cases}
\end{equation}
o\`u
\begin{equation}\label{22}
P_1(X):=\frac{X}{\nu'(\mathfrak{p})},
\qquad
\sigma_f:=1-\frac{P_1(\lambda_f(\mathfrak{p}))^2}{\mathfrak{p}},
\qquad
\nu'(\mathfrak{p}):=1+\frac{1}{\mathfrak{p}}.
\end{equation}

{\sl Remarque 2}.
Dans le cas o\`u $m'>1$, en posant $d=\mathfrak{p}^\delta$, alors:
\[\rho_{f,m'}(d)
=\begin{cases}
1 & \text {si $\delta=0$ ou $\delta\geq 1$ et ${\mathfrak{p}^{2}}\mid m'$},
\\
1-\mathfrak{p}^{-2} & \text{sinon.}
\end{cases}\]

Montrons un premier r\'esultat:

\begin{lemme}\label{lem_5}
Soient $k\ge 2$ un entier pair, $\mathfrak{p}$ un nombre premier
et $q=\mathfrak{p}^{\nu}$ avec $\nu \geq 1$.
Si $m'$ est un entier tel que $m'\mid q$ et $f\in {\rm H}_{k}^{*}(m')$, 
alors pour tout entier $r\geq 0$ la s\'erie de Dirichlet 
\[R_f(\mathfrak{p}^r, s)
:= \sum_{n\geq 1} \lambda_f(n) \lambda_f(n\mathfrak{p}^r) n^{-s}\]
v\'erifie 
\begin{equation}\label{25}
{R_f(\mathfrak{p}^r, s)=Z_f(\mathfrak{p}^r, m', s) L(s,f\otimes{f})}\end{equation}
o\`u
\begin{equation}\label{29}
L(s,f\otimes{f}):=\sum_{n\geq1} \lambda_f(n)^2 n^{-s}
\end{equation} 
et
\begin{equation}\label{26}
Z_f(\mathfrak{p}^r, m', s)
:=\begin{cases}
P_{r}(\lambda_{f}(\mathfrak{p}),s)   & \text{si $ m'=1$}
\\\noalign{\vskip 1mm}
\lambda_{f}({{\mathfrak{p}}^{r}})      & \text{sinon}
\end{cases}
\end{equation}
avec
\begin{equation}\label{27}
\begin{cases}
P_0(X,s) :=1, 
\\\noalign{\vskip 1mm}
P_1(X,s) := X/(1+\mathfrak{p}^{-s}),
\\\noalign{\vskip 1mm}
P_r(X,s):=X P_{r-1}(X,s)-P_{r-2}(X,s)
\quad(r\ge 2).
\end{cases}
\end{equation}
\end{lemme}

\begin{proof}
Si on utilise l'hypoth\`ese $m'>1$ dans l'\'egalit\'e \eqref{6}, on a\footnote{C'est ici qu'intervient la diff\'erence entre les cas $m'=1$ et $m'>1$ due \`a la multiplicit\'e des coefficients $\lambda_f(r)$. Voir la diff\'erence entre (\ref{30}) et (\ref{25}).} :
\begin{equation}\label{30}
R_{f}\left(\mathfrak{p}^{r},s\right)=\lambda_f(\mathfrak{p}^r) L(s,f\otimes{f}).
\end{equation}

Ensuite on consid\`ere le cas o\`u $m'=1$.
Si on \'ecrit chaque entier $n\geq{1}$ de fa\c con unique 
$n=n^{(\mathfrak{p})}n_{\mathfrak{p}}$ 
avec $ n_{\mathfrak{p}}\mid{{\mathfrak{p}}^{\infty}}$ et $(n^{(\mathfrak{p})}, \mathfrak{p})=1$, 
alors 
\begin{equation}\label{31}
R_f(\mathfrak{p}^r, s)
= \sum_{n\mid{{\mathfrak{p}}^{\infty}}} \lambda_f(n) \lambda_f(n\mathfrak{p}^r) n^{-s}
\sum_{(n,\mathfrak{p})=1} \lambda_f(n)^2 n^{-s}.
\end{equation}

Notons $R^*(\mathfrak{p}^r, s)$ la premi\`ere de ces deux sommes (on n'a pas indiqu\'e la d\'ependance en $f$ pour all\'eger les notations). 

Pour le cas $r=0$ :
\[R_{f}(1,s)=L(s,f\otimes{f})\] avec la notation \eqref{29}.

Quand $r=1$, on applique \eqref{5} sous la forme (avec $r=1$)
\begin{equation}\label{32}
\lambda_f(\mathfrak{p}^{k+r})
=\lambda_f(\mathfrak{p}) \lambda_f(\mathfrak{p}^{k+r-1})-\lambda_f(\mathfrak{p}^{k+r-2})
\end{equation}
pour \'ecrire
\begin{align*}
R^*(\mathfrak{p}, s)
& = \sum_{k\geq 0} \frac{\lambda_f(\mathfrak{p}^k) \lambda_f(\mathfrak{p}^{k+1})}{\mathfrak{p}^{ks}}
\\
& 
=\lambda_f(\mathfrak{p})
+\sum_{k\geq 1} \frac{\lambda_f(\mathfrak{p}^k) \lambda_f(\mathfrak{p}^{k+1})}{\mathfrak{p}^{ks}}
\\
& = \lambda_f(\mathfrak{p})
+\lambda_f(\mathfrak{p})\sum_{k\geq 1} \frac{\lambda_f(\mathfrak{p}^k)^2}{\mathfrak{p}^{ks}}
- \frac{1}{\mathfrak{p}^s} \sum_{k\geq{0}} 
\frac{\lambda_f(\mathfrak{p}^k) \lambda_f(\mathfrak{p}^{k+1})}{\mathfrak{p}^{ks}}
\\
& = \frac{\lambda_f(\mathfrak{p})}{1+\mathfrak{p}^{-s}}
\sum_{k\geq 0} \frac{\lambda_f(\mathfrak{p}^k)^2}{\mathfrak{p}^{ks}}.
\end{align*}
Avec l'\'egalit\'e (\ref{31}), on a :
\begin{equation}\label{33}
{R_{f}(\mathfrak{p},s)={\frac{\lambda_{f}(\mathfrak{p})}{{1+{{{\mathfrak{p}}^{-s}}}}}} L(s,f\otimes{f})}.
\end{equation}

Si $r\geq 2$, on utilise (\ref{32}) pour \'ecrire pour tout $k\geq 0$
$$
\sum_{k\geq 0} \frac{\lambda_f(\mathfrak{p}^k) \lambda_f(\mathfrak{p}^{k+r})}{\mathfrak{p}^{ks}}
= \lambda_f(\mathfrak{p}) \sum_{k\geq 0} 
\frac{\lambda_f(\mathfrak{p}^k) \lambda_f(\mathfrak{p}^{k+r-1})}{\mathfrak{p}^{ks}}
-\sum_{k\geq 0} 
\frac{\lambda_f(\mathfrak{p}^k) \lambda_f(\mathfrak{p}^{k+r-2})}{\mathfrak{p}^{ks}}
$$
ce qui signifie que :
\[R^{*}({\mathfrak{p}}^{r},s)=\lambda_f(\mathfrak{p}) R^{*}({\mathfrak{p}}^{r-1},s)-R^{*}({\mathfrak{p}}^{r-2},s)\]
en particulier cela donne, pour tout entier $r\geq2$ :
\begin{equation}\label{34}{ R_{f}({{\mathfrak{p}}^{r},s})=\lambda_f(\mathfrak{p}) R_{f}({\mathfrak{p}}^{r-1},s)-R_{f}({\mathfrak{p}}^{r-2},s)}.\end{equation}
Les r\'esultats pr\'ec\'edents concernant $R_{f}({{\mathfrak{p}}^{r},s})$ permettent de terminer la preuve de ce lemme.
\end{proof}

On a aussi le r\'esultat suivant :
\begin{lemme}\label{lemme 6}
Soient $k\ge 2$ un entier pair, $\mathfrak{p}$ un nombre premier
et $q=\mathfrak{p}^{\nu}$ avec $\nu \geq 1$.
On note $q=\ell m'$ et $\ell_1, \ell_2$ des entiers tels que $\ell_1, \ell_2\mid \ell$. 
Soit $f\in {\rm H}_{k}^{*}(m')$ 
alors 
\begin{equation}\label{35}
\left\langle f_{|\ell_1}, f_{|\ell_2}\right\rangle_q =\begin{cases}
\displaystyle \frac{\lambda_f(\overline{\ell})}{\sqrt{\overline{\ell}}} \left\langle f, f\right\rangle_q
& \text{si $m'>1$},
\\\noalign{\vskip 1mm}
\displaystyle \frac{P_j(\lambda_f(\mathfrak{p}))}{\sqrt{\overline{\ell}}} \left\langle f, f\right\rangle_q
& \text{si $m'=1$},
\end{cases}
\end{equation}
o\`u $\overline{\ell}:= \ell_1\ell_2/(\ell_1, \ell_2)^2={\mathfrak{p}}^{j}$, 
$P_{0}=1$, $P_{1}$ est donn\'e en (\ref{22}) et
\begin{equation}\label{37}
P_{n+2}=X P_{n+1}-P_{n}
\quad(n\geq 0).
\end{equation}
\end{lemme}

\begin{proof}
On note 
\[\Gamma_{\infty}:=\left\{
\begin{pmatrix}
1 & b
\\
0 & 1
\end{pmatrix} :
b\in \Z\right\}\]
et on consid\'ere
$$
G(s) :=  \left\langle E(z,s) f(\ell_1z) , f(\ell_2z)\right\rangle_q,
$$
o\`u la s\'erie d'Eisenstein 
\begin{equation}\label{38}
E(z,s) 
= \sum_{\gamma \in \Gamma_0(q)/\Gamma_\infty} (\im \gamma z)^s.
\end{equation}
est d\'efinie pour $z\in\HH$ et se prolonge en une fonction holomorphe si $\re s>{\frac{1}{2}}$
sauf en un p\^ole simple en 1 \cite[Lemma 3.7]{DI82}.

En utilisant la m\'ethode classique de d\'eroulement expos\'ee dans \cite[Pages 72-3]{ILS}, 
on obtient si $\ell':=\ell_1/(\ell_1, \ell_2)$, 
$\ell'':=\ell_2/(\ell_1, \ell_2)$ et $[\ell_1, \ell_2]=\ell_1\ell_2/(\ell_1, \ell_2)$ :
\begin{equation}\label{39}
G(s) = (4\pi)^{1-k-s} \Gamma(s+k-1) (\ell_1\ell_2)^{-(k-1)/2} [\ell_1, \ell_2]^{-s} R_f(\ell'\ell'', s),
\end{equation}
o\`u 
\begin{equation}\label{40}
{R_{f}(\ell'\ell'',s)
:= \sum_{n} \lambda_f(\ell'n) \lambda_f(\ell''n)n^{-s}
= \sum_{n} \lambda_f(n) \lambda_{f}(\ell'\ell''n)n^{-s}}
\end{equation}
car $(\ell', \ell'')=1$ implique que $\ell'=1$ ou $\ell''=1$.

En appliquant (\ref{25}) du Lemme \ref{lem_5}, l'\'egalit\'e (\ref{39}) devient 
\[ 
G(s) = (4\pi)^{1-k-s} \Gamma(s+k-1) (\ell_1\ell_2)^{(1-k)/2} [\ell_1, \ell_2]^{-s} Z_f(\ell'\ell'', m', s) L(s, f\otimes{f}).
\]
Cette \'egalit\'e appliqu\'ee \`a $\ell_1=\ell_2=1$ montre que 
$L(s, f\otimes{f})$ a un p\^ole simple en $s=1$ \'etant donn\'e que 
c'est le cas pour les s\'eries d'Eisenstein $E(z,s)$.
De plus (\ref{26}) et (\ref{27}) montrent que $Z_f(\ell'\ell'', m', s)$ est holomorphe en $s=1$, on posera $ Z_f(\ell'\ell'', m') = Z_f(\ell'\ell'', m', 1)$. 

On passe alors aux r\'esidus en $s=1$, pour cela, rappelons la  formule classique \cite{DI82} 
qui concerne les s\'eries d'Eisenstein :
\[\mathop{\rm Res}_{s=1} E(z,s)={\frac{3}{\pi\nu(q)}}\]
qui montre que ce r\'esidu $r$ est ind\'ependant de $z$.

On trouve donc 
\[ 
r\left\langle f(\ell_1z) , f(\ell_2z)\right\rangle_q
= \frac{\Gamma(k)}{(4\pi)^k} \frac{(\ell_1\ell_2)^{-(k-1)/2}}{[\ell_1, \ell_2]} Z_f(\ell'\ell'', m') 
\mathop{\rm Res}_{s=1} L(s, f\otimes{f}) 
\]
ce qui s'\'ecrit encore \`a l'aide de (\ref{16}) :
\begin{equation}\label{41}
r\left\langle f_{|\ell_1}, f_{|\ell_2}\right\rangle_q 
= \frac{\Gamma(k)}{(4\pi)^k}
\frac{(\ell_1\ell_2)^{1/2}}{[\ell_1, \ell_2]} Z_f(\ell'\ell'', m') \mathop{\rm Res}_{s=1}L(s, f\otimes{f}).
\end{equation}
Le cas $\ell_1=\ell_2=1$ donne
\begin{equation}\label{42}
r \left\langle f, f\right\rangle_q 
= \frac{\Gamma(k)}{(4\pi)^k} \mathop{\rm Res}_{s=1}L(s,f\otimes{f}).
\end{equation}

Enfin les \'egalit\'es (\ref{41}) et (\ref{42}) donnent
\begin{equation}\label{43}{ \left\langle f_{|\ell_1}, f_{|\ell_2}\right\rangle_{q} ={\frac{Z_{f}(\overline{\ell},m')} {\sqrt{\overline{\ell}} }} \left\langle f, f\right\rangle_{q}}.\end{equation}
Mais puisqu'on a :
\[Z_{f}({\mathfrak{p}}^{r},m')
=\begin{cases}
P_{r}(\lambda_f(\mathfrak{p}) ,1)  & \text {si $ m'=1$ }\\
\lambda_{f}({{\mathfrak{p}}^{r}}) & \text{ sinon}
\end{cases}\]
en posant 
\[P_{r}(X)=P_{r}(X,1)\]
on retrouve (\ref{35}) et (\ref{37}) gr\^ace aux relations (\ref{43}) et (\ref{27}).
\end{proof}

On aura aussi besoin d'un autre r\'esultat :
\begin{lemme}\label{lem_7}
Si $f\in {\rm H}_{k}^{*}(1)$, on a les \'egalit\'es suivantes :
\begin{align}
\left\langle f_{\mathfrak{p}^{r+1}}, f_1\right\rangle_q
& = \left\langle f_{\mathfrak{p}^{r+1}}, f_{\mathfrak{p}}\right\rangle_q=0
\quad(r\geq 1),
\label{44}
\\
\left\langle f_{\mathfrak{p}^{r+1}}, f_{\mathfrak{p}^2}\right\rangle_q
& =0
\hskip 28,5mm
(r\geq 2),
\label{44a}
\\
\left\langle f_{\mathfrak{p}^{r+1}}, f_{\mathfrak{p}^j}\right\rangle_q
& = \left\langle f_{\mathfrak{p}^r}, f_{\mathfrak{p}^{j-1}}\right\rangle_q
\hskip 9,5mm
(3\leq j\leq r).
\label{45}
\end{align}
\end{lemme}

\begin{proof}
En ce qui concerne (\ref{44}), ${\left\langle f_{{\mathfrak{p}}^{r+1}},{f_{1}}\right\rangle}_q$ vaut \`a un facteur multiplicatif pr\`es :
\[{\left\langle {f_{|{\mathfrak{p}}^{r+1}}}-\nu'(\mathfrak{p}){{\frac{P_{1}\left(\lambda_{f}\left({\mathfrak{p}}\right)\right)}{\sqrt{{\mathfrak{p}}}}}}f_{|{\mathfrak{p}}^{r}}+{\frac{1}{\mathfrak{p}}}f_{|{\mathfrak{p}}^{r-1}},{f_{|1}}\right\rangle}_q\]
qui vaut avec (\ref{35})
\[{{\frac{P_{r+1}\left(\lambda_{f}\left({\mathfrak{p}}\right)\right)}{\sqrt{{\mathfrak{p}}^{r+1}}}}}
-{{\frac{\nu'{P_{1}\left(\lambda_{f}\left({\mathfrak{p}}\right)\right)}P_{r}\left(\lambda_{f}\left({\mathfrak{p}}\right)\right)}{\sqrt{{\mathfrak{p}}^{r+1}}}}}+{{\frac{P_{r-1}\left(\lambda_{f}\left({\mathfrak{p}}\right)\right)}{\sqrt{{\mathfrak{p}}^{r+1}}}}}
\]
ce qui vaut aussi (\`a un facteur multiplicatif pr\`es) :
\[P_{r+1}\left(\lambda_{f}\left({\mathfrak{p}}\right)\right)
-\nu'\left({P_{1}\left(\lambda_{f}\left({\mathfrak{p}}\right)\right)}{P_{r}\left(\lambda_{f}\left({\mathfrak{p}}\right)\right)}+{P_{r-1}\left(\lambda_{f}\left({\mathfrak{p}}\right)\right)}\right).\]
Mais la r\'ecurrence (\ref{37}) donne 
$P_{r+1}=\nu'P_{1}P_{r}-P_{r-1}$.
Donc on a bien ${\left\langle f_{{\mathfrak{p}}^{r+1}},{f_{1}}\right\rangle}_q=0$.

Pour ce qui concerne ${\left\langle f_{{\mathfrak{p}}^{r+1}},{f_{\mathfrak{p}}}\right\rangle}_q$ et
${\left\langle f_{{\mathfrak{p}}^{r+1}},f_{{{\mathfrak{p}^{2}}}}\right\rangle}_q$ les calculs sont similaires et la r\'ecurrence (\ref{37}) permet d'\'etablir qu'ils sont nuls.

Passons \`a (\ref{45}) avec $3\leq{j}\leq{r}$ on a $\left\langle f_{{\mathfrak{p}}^{r+1}},f_{\mathfrak{p}^{j}}\right\rangle_q$ qui vaut :
$$
\biggl\langle%
f_{|\mathfrak{p}^{r+1}}%
-\nu'(\mathfrak{p})%
\frac{%
P_{1}\left(\lambda_{f}(\mathfrak{p})\right)%
}
{\sqrt{\mathfrak{p}}}f_{|\mathfrak{p}^{r}}
+\frac{1}{\mathfrak{p}}%
f_{|\mathfrak{p}^{r-1}}, 
f_{|\mathfrak{p}^{j}}%
-\nu'(\mathfrak{p})%
\frac{%
P_{1}\left(\lambda_{f}(\mathfrak{p})\right)%
}
{\sqrt{\mathfrak{p}}}f_{|\mathfrak{p}^{j-1}}
+\frac{1}{\mathfrak{p}}%
f_{|\mathfrak{p}^{j-2}}
\biggr\rangle_q.
$$
D'autre part 
\begin{align*}
& {\left\langle f_{{\mathfrak{p}}^{r}},f_{{{\mathfrak{p}}^{j-1}}}\right\rangle}_q
\\
& =
\biggl\langle%
f_{|\mathfrak{p}^{r}}%
-\nu'(\mathfrak{p})%
\frac{%
P_{1}\left(\lambda_{f}(\mathfrak{p})\right)%
}
{\sqrt{\mathfrak{p}}}f_{|\mathfrak{p}^{r-1}}
+\frac{f_{|\mathfrak{p}^{r-2}}}{\mathfrak{p}},%
f_{|\mathfrak{p}^{j-1}}%
-\nu'(\mathfrak{p})%
\frac{%
P_{1}\left(\lambda_{f}(\mathfrak{p})\right)%
}
{\sqrt{\mathfrak{p}}}f_{|\mathfrak{p}^{j-2}}
+\frac{f_{|\mathfrak{p}^{j-3}}}{\mathfrak{p}}%
\biggr\rangle_q.
\end{align*}
Avec (\ref{35}) on peut d\'evelopper ce produit scalaire  et on retrouve le m\^eme r\'esultat qu'avec ${\left\langle f_{{\mathfrak{p}}^{r+1}},f_{{{\mathfrak{p}}^{j}}}\right\rangle}_q$ car les indices $r$ et $j$ ont diminu\'e de $1$ mais leur diff\'erence elle reste la m\^eme, plus pr\'ecis\'ement :
\[{\big\langle f_{{|\mathfrak{p}}^{k}},f_{{{|\mathfrak{p}}^{j-k'}}}\big\rangle}_q
={\big\langle f_{{|\mathfrak{p}}^{k-1}},f_{{{|\mathfrak{p}}^{j-k'-1}}}\big\rangle}_q\]
pour $k=r-1, r$ ou $r+1$ et $k'=0, 1$ ou $2$. 
\end{proof}

\begin{lemme}\label{lem_8}
Soient $k\ge 2$ un entier pair, $\mathfrak{p}$ un nombre premier
et $q=\mathfrak{p}^{\nu}$ avec $\nu \geq 1$.
Soit $f\in {\rm H}_{k}^{*}(m')$ avec $q=\ell m'$.
\begin{itemize}
\item
Si $m'>1$, alors la famille $E^f_{\ell}:=\{f_d : d\mid \ell\}$ est une base orthogonale de l'espace $S_k(\ell, f)$ v\'erifiant $\left\|f_d \right\|_q=\left\|f\right\|_q$ pour tout d. 
\item
Si $m'=1$, alors la famille $E^{f}_{q}:=\{f_d : d\mid q\}$ est une base orthogonale de $S_{k}(q,f)$ 
v\'erifiant $\|f_d\|_q=\|f\|_q$ pour tout $d$.
 \end{itemize}
\end{lemme}

\begin{proof}
Dans un premier temps, supposons $m'>1$ et montrons l'\'egalit\'e (\ref{17}).
On pose :
\[
\delta_f(d_1, d_2)
= \frac{\left\langle f_{d_1}, f_{d_2}\right\rangle_q}{\left\langle f, f\right\rangle_q}  
\]
pour $d_{1} ,d_{2} \mid \ell$ (o\`u on rappelle que $f_d =\sum_{n\mid \ell}x_{d}(n,f)f_{|n} $).
Selon \eqref{35} on a : 
\[%
\delta_f(d_1, d_2)=%
\sum_{\ell_1, \ell_2\mid \ell}%
x_{d_1}(\ell_1, f)\overline{x}_{d_2}(\ell_2, f)\frac{\lambda_{f}(\overline{\ell})}{\sqrt{\overline{\ell}}}.
\]
\'Ecrivant $\ell_1=a\ell'$ et $\ell_2=a\ell''$ avec $a=(\ell_1, \ell_2)$, on a, \`a l'aide de (\ref{5}) et (\ref{6}) :
\begin{align*}
\delta_f(d_1,d_2)
& =\sum_{a\mid \ell} 
\sum_{\substack{\ell', \, \ell''\mid (\ell/a)\\ (\ell', \ell'')=1}} 
x_{d_1}(a\ell', f) \overline{x}_{d_2}(a\ell'', f) \frac{\lambda_f(\ell') \lambda_f(\ell'')}{\sqrt{\ell'\ell''}}
\\
& = \sum_{a\mid \ell} \sum_{b\mid (\ell/a)} \mu(b)
\sum_{\ell', \,\ell''\mid (\ell/(ab))} 
x_{d_1}(ab\ell', f) \overline{x}_{d_2}(ab\ell'', f) \frac{\lambda_f(b\ell') \lambda_f(b\ell'')}{b\sqrt{\ell'\ell''}}
\\
& = \sum_{a\mid \ell} \sum_{b\mid (\ell/a)} \frac{\mu(b)\lambda_f(b)^2}{b} \!
\sum_{\ell'\mid (\ell/(ab))} \! x_{d_1} (ab\ell', f) \frac{\lambda_f(\ell')}{\sqrt{\ell'}} \!
\sum_{\ell''\mid (\ell/(ab))} \! \overline{x}_{d_2}(ab\ell'', f) \frac{\lambda_f(\ell'')}{\sqrt{\ell''}}.
\end{align*}
En posant d\'esormais $c =ab$, on trouve :
\begin{equation}\label{47}{\delta_f(d_1, d_2) = \sum_{c\mid \ell}\rho_{f,m'}(c)
y_{d_{1}}(c,f)\overline{y}_{d_{2}}(c,f)}\end{equation}
o\`u on a not\'e :
\[
\rho_{f,m'}(c)=\sum_{n\mid c}\frac{\mu(n)\lambda_f^2(n)}{n},
\qquad
y_{d}(c,f) := \sum_{r\mid (\ell/c)} x_d(rc, f) \frac{\lambda_f(r)}{\sqrt{r}}.
\]
La formule d'inversion de M\"obius appliqu\'ee \`a l'\'egalit\'e ci-dessus donne :
\begin{equation}\label{48}
x_{d}(c, f) = \sum_{r\mid (\ell/c)} y_{d}(rc, f) \mu(r) \frac{\lambda_f(r)}{\sqrt{r}}.
\end{equation}
Pour que $E^f_{\ell}$ soit une base orthogonale de $S_k(\ell, f)$ 
il suffit (par la d\'efinition de $\delta_f(d_1, d_2)$) que $\delta_f$ soit le symbole de Kronecker, ce qui est expliqu\'e
\footnote{Notons que $y_{d}(c,f)$ existe puisque $\rho_{f,m'}(c)=\prod_{p\mid c}(1-\lambda_f(p)^2/p)$ 
implique que  $\rho_{f}(c) \in \left]0,1\right] $ \'etant donn\'e l'\'egalit\'e \eqref{7}.} par :
\[%
y_{d}(c,f)=\begin{cases}
1/\sqrt{\rho_{f,m'}(c)} & \text{si $ d=c$},
\\
0& \text{sinon}.
\end{cases}
\]
L'\'egalit\'e d\'efinissant $y_{d}(c,f)$ \'equivaut d'apr\`es 
(\ref{48}) \`a :
\[
x_{d}(c,f)=\begin{cases}
\displaystyle \frac{\mu(r)\lambda_f(r)}{\sqrt{r\rho_{f, m'}(d)}} 
& \text {si $ d=rc$},
\\\noalign{\vskip 1mm}
0                          
& \text{sinon}.
\end{cases}
\]
Ceci termine la preuve de l'othogonalit\'e dans le cas $m'>1$.

Passons \`a la preuve de la base orthogonale de $S_{k}(q,f)$.
On supposera d\'esormais $m'=1$ dans le reste de la preuve du Lemme \ref{lem_8}.
Pour v\'erifier que la famille propos\'ee en (\ref{21}) existe, 
montrons que $\sigma_{f}$ est strictement positif, en effet on sait que pour toute forme parabolique :
\[0\leq \lambda_f(\mathfrak{p})^2\leq \tau(\mathfrak{p})^2=4\]
puisque d'autre part $\mathfrak{p}\geq{2}$ alors
$9/2\leq \mathfrak{p}(1+\mathfrak{p}^{-1})^2$
ainsi on en conclut (d'apr\`es (\ref{22})) :
\begin{equation}\label{49}
\sigma_f\geq 1/9.
\end{equation}

Montrons que les formes propos\'ees ont toutes la m\^eme norme que celle de $f$.
Pour $f_{1}$ c'est imm\'ediat.
Pour $f_{\mathfrak{p}}$, d'apr\`es (\ref{21}) :
\[{\left\|f_{\mathfrak{p}}\right\|}_q^{2}={\frac{1}{\sigma_{f}}} \left({\left\|f_{|\mathfrak{p}}\right\|}_q^{2}
+{\frac{P^{2}_{1}(\lambda_f(\mathfrak{p}))}{\mathfrak{p}}}{\left\|f_{|1}\right\|}_q^{2}-2{{\frac{P_{1}(\lambda_f(\mathfrak{p}))}{\sqrt{\mathfrak{p}}}}}{\left\langle f_{|\mathfrak{p}},f_{|1}\right\rangle}_q\right).\]
Mais en utilisant (\ref{35}) on a :
\[%
\Vert f_{|\mathfrak{p}^r}\Vert_q^2=\Vert f\Vert_q^2%
\]
et
\[
{\left\langle f_{|\mathfrak{p}},f_{|1}\right\rangle}_q={{\frac{P_{1}(\lambda_f(\mathfrak{p}))}{\sqrt{\mathfrak{p}}}}}{\left\|f\right\|}_q^{2}.
\]
On trouve alors, \'etant donn\'e (\ref{22}) :
\[{\left\|f_{\mathfrak{p}}\right\|}_q^{2}={\left\|f\right\|}_q^{2}.\]

Il reste \`a traiter le cas de $f_{{\mathfrak{p}}^{r}}$ o\`u $r\geq 2$.
On notera $\nu'$ au lieu de $\nu'(\mathfrak{p})$ pour all�ger.
D'apr\`es (\ref{21}) :
\begin{align*}%
\Vert f_{\mathfrak{p}^r}\Vert_q^{2}
& = \frac{1}{(1-\mathfrak{p}^{-2})\sigma_{f}}%
\bigg(%
\Vert f_{|\mathfrak{p}^r}\Vert_q^{2}
+\frac{\nu^{'2}P^2_{1}(\lambda_{f}(\mathfrak{p}))}{\mathfrak{p}}%
\Vert f_{|\mathfrak{p}^{r-1}}\Vert_q^{2}%
+\frac{\Vert f_{|\mathfrak{p}^{r-2}}\Vert_q^2}{\mathfrak{p}^2}
\\%
& \quad
-2\nu'\frac{P_{1}(\lambda_{f}(\mathfrak{p}))}{\sqrt{\mathfrak{p}}}%
\left\langle%
f_{|\mathfrak{p}^r},f_{|\mathfrak{p}^{r-1}}%
\right\rangle_q%
+\frac{2}{\mathfrak{p}}
\left\langle%
f_{|\mathfrak{p}^r},f_{|\mathfrak{p}^{r-2}}%
\right\rangle_q
\\
& \quad
-2\nu'\frac{P_{1}(\lambda_{f}(\mathfrak{p}))}{\mathfrak{p}\sqrt{\mathfrak{p}}}%
\left\langle%
f_{|\mathfrak{p}^{r-1}},f_{|\mathfrak{p}^{r-2}}%
\right\rangle_q%
\bigg).%
\end{align*}
En utilisant (\ref{35}), on trouve que  $\Vert f_{\mathfrak{p}^r}\Vert_q^2$ vaut au facteur multiplicatif 
$\Vert f\Vert_q^2$ pr\`es :
$$
\frac{1}{(1-\mathfrak{p}^{-2})\sigma_f}
\left(1+\frac{\nu'(\nu'-2)P_1(\lambda_f(\mathfrak{p}))^2}{\mathfrak{p}}
+\frac{1+2P_2(\lambda_f(\mathfrak{p}))-2\nu'P_1(\lambda_f(\mathfrak{p}))^2}{\mathfrak{p}^2}\right).
$$
Utilisons la relation de r\'ecurrence (\ref{37}) qui peut se r\'e\'ecrire (\`a l'aide de (\ref{22})) :
\[P_{2}=\nu'P^{2}_{1}-1\]
pour transformer le terme pr\'ec\'edent en :
\[\frac{1}{{\sigma_{f}}({1-{{\mathfrak{p}}^{-2}}})} {{\left(1-{\frac{P^{2}_{1}(\lambda_f(\mathfrak{p}))}{\mathfrak{p}}}\right)}\left({1-{\frac{1}{{\mathfrak{p}}^{2}}}}\right)}\]
ce qui donne bien $\|f_{\mathfrak{p}^r}\|_q=\|f\|_q$.
pour tout $r\geq{2}$.

Montrons maintenant par r\'ecurrence sur $r\geq{1}$ 
que $\langle f_{\mathfrak{p}^r}, f_{\mathfrak{p}^k}\rangle_q=0$
pour tout $k<r$.
Pour $r=1$ : ${\left\langle f_{{\mathfrak{p}}},f_{1}\right\rangle}_q$
vaut \`a un facteur multiplicatif pr\`es :
\[{\left\langle f_{|{\mathfrak{p}}}-{\frac{P_{1}(\lambda_f(\mathfrak{p}))}{\sqrt{\mathfrak{p}}}} f_{|1},f_{|1}\right\rangle}_q\]
ce qui vaut selon (\ref{35})
\[{\frac{P_{1}(\lambda_f(\mathfrak{p}))}{\sqrt{\mathfrak{p}}}}{{\left\|f\right\|}_q}^{2}-
{\frac{P_{1}(\lambda_f(\mathfrak{p}))}{\sqrt{\mathfrak{p}}}}{{\left\|f_{|1}\right\|}}_q^{2}=0.\]
On suppose l'orthogonalit\'e vraie jusqu'\`a $r$ et montrons que c'est le cas en $r+1$. 
Appliquons alors le Lemme \ref{lem_7}. 
Ce r\'esultat permet de terminer la r\'ecurrence car le cas $r+1$ peut lui-m\^eme se traiter par r\'ecurrence sur $j\leq r$ en montrant que ${\left\langle f_{{\mathfrak{p}}^{r+1}},f_{{{\mathfrak{p}}^{j}}}\right\rangle}_q=0$, on en conclut donc que la famille propos\'ee 
dans le Lemme \ref{lem_8} est orthogonale.
\end{proof}

{\sl Remarque 3}. 
Il est \`a noter que le choix que l'on a fait de $ y_{d} $ implique que la norme de tous les $ f_d  $ est la m\^eme et plus particuli\`erement : ${\left\|f_d \right\|}_{q}={\left\|f\right\|}_{q}$ .

{\sl Remarque 4}. 
Il est important de noter que la d\'emonstration du cas o\`u $m'>1$ devient fause si $m'=1$. Il est alors plus difficile de d\'ecrire une base orthogonale de $S_{k}(q,f)$ (voir ce qui pr\'ec\`ede).

Toujours dans le but d'exprimer $\Delta_q(m, n)$ en fonction de $ \Delta^*_q(m, n) $, on aura recours au r\'esultat suivant :

\begin{lemme}\label{lem_9}
Soient $k\ge 2$ un entier pair, $\mathfrak{p}$ un nombre premier
et $q=\mathfrak{p}^{\nu}$ avec $\nu \geq 1$.
Soit $f\in {\rm H}_{k}^{*}(m')$ avec $q=\ell m'$.
Alors
\[\omega_{q}(f)=\begin{cases}
\displaystyle {\frac{\omega_{m'}(f)}{\ell}}   & \text {si $m'>1$ }
\\\noalign{\vskip 2mm}
\displaystyle{\frac{\omega_{1}(f)}{\nu(q)}} & \text{si $m'=1$}
\end{cases}\]
o\`u on a not\'e
\[\nu(n)=n\prod_{p\mid n}\left(1+{\frac{1}{p}}\right).\]
\end{lemme}

\begin{proof}
Puisque $\Gamma_0(q)$ et $\Gamma_0(m')$ sont des sous-groupes de $SL(2,\Z)$ d'indices respectifs $\nu(q)$ et $\nu(m')$ (voir \cite[p 35]{Iw 97}), en utilisant la formule de multiplicit\'e des indices, on obtient l'indice suivant : \[[\Gamma_0(m') ; \Gamma_0(q)]={{\frac{\nu(q)}{\nu(m')}}}=\begin{cases}
\ell       & \text {si $m'>1$}
\\
\nu(q)  & \text{si $m'=1$}.
\end{cases}
\]

Notons $F'$ un domaine fondamental de $\Gamma_0(m')$. 
Prenons $\left\{\sigma_{j} , j\in J\right\}$ un ensemble de repr\'esentants de $\Gamma_0(m')/\Gamma_0(q)$ (d'apr\`es ce qui pr\'ec\`ede
$J$ est de cardinal $\nu(q)/\nu(m')$); 
d'apr\`es \cite[p 32]{Iw 97}, 
on a $\bigcup_{j\in J}\sigma_{j}(F')$ est un domaine fondamental de $\Gamma_0(q)$ ainsi :
\begin{align*}
\|f\|_q^2
& = \int_{\bigcup_{j\in J}\sigma_{j}(F')} |f(z)|^2 y^k \frac{{\rm d}x {\rm d}y}{y^2}
\\
& =\sum_{j\in J}\int_{\sigma_{j}(F')} {\left|f(z)\right|}^{2}{{y^k}}{\frac{{\rm d}x {\rm d}y}{y^{2}}}
\end{align*}
avec les changements de variables $z\mapsto\sigma^{-1}_{j}(z)$, on obtient, 
puisque $y^{-2} {\rm d}x {\rm d}y$ est $SL(2,\R)$-invariante :
\[{\left\|f\right\|}_{q}^{2}
=\sum_{j\in J}\int_{F'} {\left|f(\sigma_{j}z)\right|}^{2}{{{\left({\im\sigma_{j}(z)}\right)}^k}}{\frac{{\rm d}x {\rm d}y}{y^{2}}}.\]
On utilise alors la relation (\ref{relationmodulaire}) pour obtenir :
$$
\|f\|_q^2
={\rm Card}(J) \int_{F'} |f(z)|^2 y^k \frac{{\rm d}x {\rm d}y}{y^2}.
$$
Connaissant d\'esormais la valeur de ${\rm Card}(J)$, on obtient :
\[
\|f\|_q^2 = \frac{\nu(q)}{\nu(m')} \|f\|_{m'}^2.
\]
Ce qui donne enfin (\`a l'aide de l'\'egalit\'e (\ref{4})) :
\[
\omega_q(f)=\begin{cases}
\displaystyle \frac{\omega_{m'}(f)}{\ell}  & \text {si $ m'>1$},
\\\noalign{\vskip 2mm}
\displaystyle\frac{\omega_1(f)}{\nu(q)}  & \text{si $m'=1$}.
\end{cases}
\]
Cela ach\`eve la d\'emonstration.
\end{proof}

Le r\'esultat suivant sera \'egalement utile:

\begin{lemme}\label{lem_10}
Soient $k\ge 2$ un entier pair, $\mathfrak{p}$ un nombre premier
et $q=\mathfrak{p}^{\nu}$ avec $\nu \geq 1$.
Soit $f\in {\rm H}_{k}^{*}(m')$ avec $q=\ell m'$.
Alors
les coefficients $x_{d}(1, f)$ (pour $d\mid \ell$) d\'efinis de \eqref{17} \`a \eqref{22} v\'erifient
\begin{align}
\sum_{d\mid \ell} x_d(1,f)^2
& =\left(1-\mu(m')^2/\mathfrak{p}^2\right)^{-\omega(\ell)}
\quad(m'>1),
\label{50}
\\
\sum_{d\mid q}x_d(1,f)^2
& \ll 1
\quad(m'=1).
\label{51}
\end{align}
La constante impliqu\'ee est absolue.
\end{lemme}

\begin{proof}
Commen\c cons par le cas $m'>1$.
D'apr\`es (\ref{17})-(\ref{18}) et le fait que $\ell\mid \mathfrak{p}^\infty$, on peut \'ecrire  
\begin{equation}\label{52}
\begin{aligned}
\sum_{d\mid \ell} x_d(1,f)^2
& = \sum_{d\mid \ell} \frac{\mu(d)^2 \lambda_f(d)^2}{d\rho_{f, m'}(d)}
= \left(1+\frac{\lambda_f(\mathfrak{p})^2}{\mathfrak{p}\rho_{f,m'}(\mathfrak{p})}\right)^{\omega(\ell)}
\\
& = \frac{1}{\rho_{f, m'}(\mathfrak{p})^{\omega(\ell)}}
= \frac{1}{\rho_{f, m'}(\ell)}.
\end{aligned}
\end{equation}
D'autre part, les relations (\ref{18}) et (\ref{7}) nous permettent d'\'ecrire
\begin{equation}\label{54}
\rho_{f, m'}(\ell)=\left(1-\frac{\mu(m')^2}{\mathfrak{p}^2}\right)^{\omega(\ell)}
\qquad
(m'>1).
\end{equation}
En conclusion, d'apr\`es (\ref{52}) et (\ref{54}), si $m'>1$ :
\begin{equation}\label{55}{\sum_{d\mid \ell}x_d(1,f)^2={\left(1-{\frac{\mu(m')^2}{\mathfrak{p}^{2}}}\right)}^{-\omega(\ell)}}.
\end{equation}

Passons au cas o\`u $m'=1$. 
D'apr\`es (\ref{21}), $x_{{\mathfrak{p}}^{r}}(1,f)=0$ pour tout $r\geq{3}$ et donc :
\[\sum_{d\mid q}x_d(1,f)^2=\sum_{d\mid \mathfrak{p}}x_d(1,f)^2+\left(x_{{\mathfrak{p}}^{2}}(1,f)\right)\]
le terme entre parenth\`eses n'existant que si $q\geq{\mathfrak{p}}^{2}$.

D'apr\`es (\ref{21}) et avec les notations pr\'ec\'edentes, on obtient
\[\sum_{d\mid q}x_d(1,f)^2=
1+{{\frac{P^{2}_{1}(\lambda_{f}(\mathfrak{p}))}{\mathfrak{p}}\sigma_{f}}}+\left({\frac{1}{{{\mathfrak{p}}^{2}}(1-{\mathfrak{p}}^{-2})\sigma_{f}}}\right).\]
\`A l'aide de l'expression (\ref{22}) de $\sigma_{f}$, on a :
\[\sum_{d\mid q}x_d(1,f)^2=
{\frac{1}{\sigma_{f}}}+\left({\frac{1}{({\mathfrak{p}}^{2}-1)\sigma_{f}}}\right)\]
ce qui donne si $m'=1$ :
\begin{equation}\label{56}
\sum_{d\mid q} x_d(1,f)^2
=\begin{cases}
\displaystyle \frac{1}{\sigma_f} & \text {si $q=\mathfrak{p}$},
\\\noalign{\vskip 2mm}
\displaystyle \frac{1}{(1-\mathfrak{p}^{-2})\sigma_f} & \text{si $q\geq \mathfrak{p}^2$}.
\end{cases}
\end{equation}
Ce r\'esultat et la minoration (\ref{49}) montrent que si $m'=1$ alors
\begin{equation}\label{57}
{\sum_{d\mid q}x_d(1,f)^2\ll{1}}. 
\end{equation}
Cela ach\`eve la d\'emonstration.
\end{proof}

On en vient au r\'esultat liant $\Delta_q(m, n)$ et $ \Delta^*_q(m, n) $ :

\begin{lemme}\label{lem_11}
Soient $k\ge 2$ un entier pair, $\mathfrak{p}$ un nombre premier
et $q=\mathfrak{p}^{\nu}$ avec $\nu \geq 1$.
Alors pour tous entiers $m\ge 1$ et $n\ge 1$ tels que $\mathfrak{p}\nmid mn$, on a
\begin{align}
\Delta_q(m, n) 
& = \sum_{\substack{q=\ell m'\\ m'>1}} \frac{1}{\ell} 
\bigg(1-\frac{\mu(m')^2}{\mathfrak{p}^2}\bigg)^{-\omega(\ell)}\Delta^{*}_{m'}(m, n)+O\bigg(\frac{\tau(m)\tau(n)}{q}\bigg),
\label{59}
\\
\Delta^*_q(m, n)
& = \sum_{\substack{q=\ell m'\\ m'>1}} \mu(\ell)
\bigg(\mathfrak{p}-\frac{\mu(m')^2}{\mathfrak{p}}\bigg)^{-\omega(\ell)} \Delta_{m'}(m,n)+O\bigg(\frac{\tau(m)\tau(n)}{q}\bigg).
\label{60}
\end{align}
Les constantes impliqu\'ees sont absolues.
\end{lemme}

\begin{proof}
Rappelons que si $q=\ell m'$, si $\ell$ est un diviseur de $q$, 
$d$ un diviseur de $\ell$ et qu'on a $f\in {\rm H}_{k}^{*}(m')$ alors
\[ f_d =\sum_{c\mid \ell} x_{d}(c,f) f_{|c} \] 
donc
\[f_d (z)=\sum_{c\mid \ell} {{c}^{k/2}} x_{d}(c,f) f(cz).\] 
Rappelons que $a_{g}(j)$ d\'esigne le $j$-\`eme coefficient de Fourier d'une forme parabolique $g$, on a alors 
$$
a_{f_d}(j)=\sum_{\substack{c\mid \ell\\ j=rc}} c^{k/2} x_{d}(c,f) a_f(r).
$$  
Ainsi si $\mathfrak{p}\nmid j$ alors $a_{f_d }(j)= x_d(1, f) a_f(j)$
et avec (\ref{2}) on a (puisque  $\mathfrak{p}\nmid mn $)
\begin{equation}\label{61}{{{\lambda}_{f_d }}(j)= x_{d}(1,f) {\lambda}_{f}(j)}
\quad(j=m, n).
\end{equation}

Utilisons maintenant la relation (\ref{15}), 
avec les termes $f_d $ d\'esignant les \'el\'ements d'une base orthogonale de $S_k(\ell, f)$ 
dans le sens du Lemme \ref{lem_8} :
\[\Delta_q(m, n)=\sum_{\substack{q=\ell m'}}\sum_{f\in{{\rm H}^{*}_{k}(m')}}\sum_{d\mid \ell}\omega_{q}(f_d ){{\lambda}_{f_d }}(m){{\lambda}_{f_d }}(n).\]
\'Etant donn\'e que le Lemme \ref{lem_8} donne 
$\left\|f_d \right\|_q=\left\|f\right\|_q$ pour tout $d\geq{1}$ 
alors d'apr\`es (\ref{4}), on  a pour tout $d\geq{1}$ :
$\omega_{q}(f_d )=\omega_{q}(f)$.
D'apr\`es (\ref{61}), on a :
\begin{align*}
\Delta_q(m, n)
& =\sum_{q=\ell m'} \sum_{f\in {\rm H}^{*}_k(m')} \omega_q(f)
\sum_{d\mid \ell} \lambda_{f_d}(m) \lambda_{f_d}(n)
\\
& =\sum_{q=\ell m'} \sum_{f\in {\rm H}^{*}_k(m')} \omega_q(f) \lambda_f(m) \lambda_f(n)
\sum_{d\mid \ell} x_d(1,f)^2.
\end{align*}

D'apr\`es le Lemme \ref{lem_8}, 
les coefficients $x_{d}(c,f)$ diff\`erent que $m'$ soit \'egal ou pas \`a $1$, 
on va donc distinguer les 2 cas dans le calcul de $\Delta_{q}$ :
\begin{equation}\label{62}
\begin{aligned}
\Delta_q(m, n)
& = \sum_{\substack{q=\ell m'\\m'>1}}
\sum_{f\in {\rm H}^{*}_{k}(m')} 
\omega_q(f) \lambda_f(m)
\lambda_f(n)
\sum_{d\mid \ell} x_d(1,f)^2
\\
& \quad
+ \sum_{f\in {\rm H}^{*}_{k}(1)} \omega_{q}(f) \lambda_f(m) \lambda_f(n)
\sum_{d\mid q} x_d(1,f)^2.
\end{aligned}
\end{equation}

On utilise alors le Lemme \ref{lem_10} qui permet d'\'ecrire:
\begin{equation}\label{63}
\begin{aligned}
\Delta_q(m, n)
& =\sum_{\substack{q=\ell m'\\m'>1}}{{\left(1-{\frac{\mu(m')^2}{\mathfrak{p}^{2}}}\right)}^{-\omega(\ell)}}
\sum_{f\in{{\rm H}^{*}_{k}(m')}}\omega_{q}(f){{\lambda}_{f}}(m){{\lambda}_{f}}(n)
\\
& \quad
+\sum_{f\in{{\rm H}^{*}_{k}(1)}}\omega_{q}(f){{\lambda}_{f}}(m){{\lambda}_{f}}(n)\sum_{d\mid q}x_d(1,f)^2.
\end{aligned}
\end{equation}

Afin de faire apparaitre dans (\ref{63}) les nombres $\Delta^{*}_{m'}$ exprimons $\omega_{q}(f)$ 
en fonction de $\omega_{m'}(f)$,
on a alors recours au Lemme \ref{lem_9},
ce r\'esultat appliqu\'e \`a (\ref{63}), donne :
\begin{equation}\label{64}
\begin{aligned}
\Delta_q(m, n)
& =\sum_{\substack{q=\ell m'\\m'>1}}
{\frac{1}{\ell}}{{\left(1-{\frac{\mu(m')^2}{\mathfrak{p}^{2}}}\right)}^{-\omega(\ell)}}
\sum_{f\in{{\rm H}^{*}_{k}(m')}}\omega_{m'}(f){{\lambda}_{f}}(m){{\lambda}_{f}}(n)
\\
& \quad
+ \frac{1}{\nu(q)} \sum_{f\in {\rm H}^{*}_k(1)} \omega_{1}(f) \lambda_f(m) \lambda_f(n)
\sum_{d\mid q}x_d(1,f)^2
\\
& =\sum_{\substack{q=\ell m'\\m'>1}}{\frac{1}{\ell}}{{\left(1-{\frac{\mu(m')^2}{\mathfrak{p}^{2}}}\right)}^{-\omega(\ell)}}\Delta^{*}_{m'}(m,n)
\\
& \quad
+{\frac{1}{\nu(q)}}\sum_{f\in{{\rm H}^{*}_{k}(1)}}\omega_{1}(f){{\lambda}_{f}}(m){{\lambda}_{f}}(n)\sum_{d\mid q}x_d(1,f)^2.
\end{aligned}
\end{equation}

Il reste \`a majorer la valeur absolue du dernier terme de cette \'egalit\'e.
Pour cela on utlise les majorations classiques $|\lambda_{f}(j)|\leq \tau(j)$ pour $j\ge 1$.
De plus avec la majoration absolue (\ref{57}), l'\'egalit\'e (\ref{64}) devient
\begin{equation}\label{65}
\begin{aligned}
\Delta_q(m, n)
& =\sum_{\substack{q=\ell m'\\ m'>1}} \frac{1}{\ell} \left(1-\frac{\mu(m')^2}{\mathfrak{p}^2}\right)^{-\omega(\ell)}\Delta^*_{m'}(m,n)
\\
& \quad
+O\bigg(\frac{\tau(m)\tau(n)}{\nu(q)} \sum_{f\in {\rm H}^*_k(1)} \omega_{1}(f)\bigg).
\end{aligned}
\end{equation}

Pour calculer la derni\`ere somme on utilise le Lemme \ref{lem_4} appliqu\'e au cas $q=m=n=1$ 
ce qui donne
\[
\sum_{f\in {\rm H}^*_k(1)} \omega_1(f)=1+O(k^{-4/3})
\]
Les deux \'egalit\'es pr\'ec\'edentes donnent bien l'\'egalit\'e (\ref{59}).
Par inversion de M\"obius, il est rapide de v\'erifier (\ref{60}).
Ceci termine la preuve du Lemme \ref{lem_11}.
\end{proof}

\subsection{Fin de la preuve du Th\'eor\`eme \ref{thm_2}}
D'abord on traite le cas o\`u $\mathfrak{p}\mid mn$ et $\nu\geq 1$.
Sans perte de g\'en\'eralit\'e, on peut supposer que $\mathfrak{p}\mid m$.
\`A l'aide de (\ref{6}) et (\ref{7}), on voit que
$$
\lambda_f(m)
=\lambda_f(\mathfrak{p}) \lambda_f(m/\mathfrak{p})
= 0.
$$
Ainsi par la d\'efinition de $\Delta^{*}_{q} $, 
on a $\Delta^*_q(m, n)=0$.

Ensuite on suppose que $\mathfrak{p}\nmid mn$ et $\nu\geq2$.
En reportant (\ref{14}) dans (\ref{60}) et en remarquant que
$$
\sum_{\substack{\ell m'=q\\ m'>1}} \mu(\ell) \bigg(\mathfrak{p}-\frac{\mu(m')^2}{\mathfrak{p}}\bigg)^{-\omega(\ell)}
= \phi(\nu, \mathfrak{p}),
$$
on obtient
$$
\Delta^*_q(m, n) 
= \phi(\mathfrak{p}, \nu)\delta_{m, n}+O\bigg(\frac{\tau(m)\tau(n)}{q}\bigg)+{\mathscr R}_1,
$$
o\`u 
\begin{align*}
{\mathscr R}_1
& \ll \frac{\sqrt{mn}\{\log(2(m,n))\}^2}{k^{4/3}}
\sum_{\substack{\ell m'=q\\ m'>1}} \frac{|\mu(\ell)|}{{m'}^{3/2}}
\bigg(\mathfrak{p}-\frac{\mu(m')^2}{\mathfrak{p}}\bigg)^{-\omega(\ell)} 
\\
& \ll \frac{\sqrt{mn\mathfrak{p}^{1-\delta_{\nu, 1}}}\{\log(2(m,n))\}^2}{k^{4/3}q^{3/2}}.
\end{align*}
Cela ach\`eve la d\'emonstration.

\section{Lemmes auxiliaires}

Soient $\zeta(s)$ la fonction de Riemann et
\begin{equation}\label{zetaq}
\zeta^{(q)}(s)
:= \zeta(s) \prod_{p\mid q}\big(1-p^{-s}\big).
\end{equation}

\subsection {Fonctions $U(y)$ et $T(y)$} 
Soit $G$ est un polyn\^ome pair de degr\'e $\geq 2$ tel que :
\begin{equation}\label{defG}
G(0)=1
\qquad{\rm et}\qquad
G(-1)=G(-2)=0.
\end{equation}
Pour $y>0$, on d\'efinit
\begin{align}
T(y)
& :=\frac{1}{2\pi i} \int_{(2)} \frac{\Gamma(s+k/2)}{\Gamma(k/2)} \frac{G(s)}{s}y^{-s}  {\rm d}s,
\label{defT}
\\
U(y)
& :=\frac{1}{{\rm i}\pi} 
\int_{(2)} \zeta^{(q)}(1+2s) \frac{\Gamma(s+k/2)^2}{\Gamma(k/2)^2} \frac{G(s)^2}{s} y^{-s} {\rm d}s.
\label{defU}
\end{align}

\begin{lemme}\label{TU}
Sous les notations pr\'ec\'edentes, on a
\begin{equation}\label{PT}
\begin{cases}
T(y) = 1+O_k(y)
& \text{si $y\to 0$},
\\\noalign{\vskip 1,5mm}
T(y) \ll_{j, k}y^{-j}
& \text{si $y\to \infty$},
\end{cases}
\end{equation}
et
\begin{equation}\label{PU}
\begin{cases}
U(y) = \displaystyle\frac{\varphi(q)}{q}
\bigg\{\log\frac{1}{y} + g_k(\mathfrak{p})+ O_k(y)\bigg\}
& \text{si $y\to 0$},
\\\noalign{\vskip 1,5mm}
U(y) \ll_{j, k}y^{-j}
& \text{si $y\to \infty$},
\end{cases}
\end{equation}
pour tout $j$ r\'eel $>0$, o\`u 
\begin{equation}\label{74}
g_k(\mathfrak{p})
:= 2\bigg(\frac{\log\mathfrak{p}}{\mathfrak{p}-1}+\frac{\Gamma'}{\Gamma}(k/2)+\gamma\bigg)
\end{equation}
et $\gamma$ est la constante d'Euler.
\end{lemme}

\begin{proof}
On ne va d\'emontrer que la formule asymptotique pour $U(y)$ quand $y\to 0$.
Les autres peuvent \^etre trouv\'ees dans \cite[Paragraphe 2.4]{KM00}.
En d\'esignant par $\zeta(s)$ la fonction de Riemann, on a
\begin{equation}\label{zeta1s}
\zeta(1+s)
= \frac{1}{s}+\sum_{0\le i\le 2} \frac{(-1)^i}{i!} \gamma_i s^i+O(s^3),
\end{equation}
o\`u $\gamma_i$ d\'esignent les constantes de Stieltjes.
\footnote{Ces nombres sont d\'efinis par 
$$
\gamma_i
:= \lim_{n\to\infty} \sum_{k=1}^{n} \bigg(\frac{(\log k)^i}{k}-\frac{(\log n)^{i+1}}{i+1}\bigg).
$$
En particulier $\gamma_0=\gamma$.}
D'autre part, on peut \'ecrire
\begin{equation}\label{p1s}
1-\mathfrak{p}^{-(1+s)}
= \frac{\varphi(q)}{q} \bigg\{
1+\sum_{1\le j\le 3} \frac{(-1)^{j+1}}{j!} \frac{(\log\mathfrak{p})^j}{\mathfrak{p}-1} s^j + O(s^4)
\bigg\}.
\end{equation}
Donc
\begin{align*}
\zeta^{(q)}(1+2s)
& = \big(1-\mathfrak{p}^{-(1+2s)}\big) \zeta(1+2s)
\\
& = \frac{\varphi(q)}{q} \bigg(\frac{1}{2s}+\frac{\log\mathfrak{p}}{\mathfrak{p}-1}+\gamma+O(s)\bigg).
\end{align*}
Ceci implique la formule annonc\'ee.
\end{proof}

\subsection {Lemme interm\'ediaire} 
Nous aurons besoin des estimations suivantes dans le calcul du troisi\`eme moment.

\begin{lemme}\label{lem_estimation}
Soient $i, j\in\N$ et $\theta>1$. 
On a
\begin{align}
\sum_{n\leq x} \tau(n)^i(\log n)^j
& = C_i x (\log x)^{2^i+j-1}
+ O\big(x(\log x)^{2^i+j-2}\big),
\label{estimation2}
\\
\sum_{n\leq x} \frac{\tau(n)^i(\log n)^j}{\sqrt{n}}
& = 2C_i \sqrt{x}(\log x)^{2^i+j-1}
+ O\big(\sqrt{x}(\log x)^{2^i+j-2}\big),
\label{estimation3}
\\
\sum_{n>x} \frac{\tau(n)^i(\log n)^j}{n^\theta}
& \ll \frac{(\log x)^{2^i+j-1}}{x^{\theta-1}}
\label{estimation4}
\end{align}
uniform\'ement pour $x\geq 3$, o\`u $C_i$ est une constante absolue.
\end{lemme}

\begin{proof}
En utilisant la formule asymptotique
$$
D_i(t):=\sum_{n\le t} \tau(n)^i
= C_i t(\log t)^{2^i-1}+O\big(t(\log t)^{2^i-2}\big),
$$ 
une simple int\'egration par parties nous donne 
\begin{align*}
\sum_{n\leq x} \tau(n)^i (\log n)^j
& = \int_{1-}^x (\log t)^j {\rm d}D_i(t)
\\
& = (\log x)^j D_i(x)
- j \int_1^x \frac{(\log t)^{j-1}}{t} D_i(t) {\rm d}t
\\\noalign{\vskip 3mm}
& = C_i x(\log x)^{2^i+j-1}
+O\big(x(\log x)^{2^i+j-2}\big).
\end{align*}
L'estimation (\ref{estimation3}) peut \^etre d\'emontr\'ee par la m\^eme m\'ethode.

De m\^eme, on a
\begin{align*}
\sum_{n>x} \frac{\tau(n)^i (\log n)^j}{n^\theta}
& = \int_{x}^\infty \frac{(\log t)^j}{t^\theta} {\rm d}D_i(t)
\\
& = -\frac{(\log x)^j}{x^\theta} D_i(x)
+ \int_x^\infty \frac{\theta (\log t)^{j}-j(\log t)^{j-1}}{t^{\theta+1}} D_i(t) {\rm d}t
\\\noalign{\vskip 1mm}
& \ll \frac{(\log x)^{2^i+j-1}}{x^{\theta-1}}.
\end{align*}
Cela ach\`eve la d\'emonstration.
\end{proof}

\section{Calcul du deuxi\`eme moment}

Le but de ce paragraphe est de calculer le deuxi\`eme moment $M_2$, d\'efini en (\ref{defMr}).
Notre r\'esultat est un peu plus g\'en\'eral.
En posant
\begin{equation}\label{defXm}
M_{r, m}=\sumh_{f\in {\rm H}_{k}^{*}(q)} \lambda_f(m) L({\textstyle\frac{1}{2}}, f)^r,
\end{equation}
nous avons le r\'esultat suivant.

\begin{proposition}\label{pro_12}
Soient $0<\eta<1$, $k\ge 2$ un entier pair, $\mathfrak{p}$ un nombre premier
et $q=\mathfrak{p}^{\nu}$ avec $\nu \geq 3$.
Pour tout $1\leq m \leq q^\eta$ et  $\mathfrak{p}\nmid m$, on a
$$
M_{2, m}
= \frac{\tau(m)}{\sqrt{m}} \bigg(\frac{\varphi(q)}{q}\bigg)^2 
\bigg\{\log\left(\frac{\hat{q}^2}{m}\right) + g_k(\mathfrak{p})\bigg\}
+ O_{k, \mathfrak{p}}\big(q^{-(1-\eta)/2}(\log q)^4\big),
$$
o\`u $g_k(\mathfrak{p})$ est d\'efinie en (\ref{74}).
En particulier
$$
M_2
= \bigg(\frac{\varphi(q)}{q}\bigg)^2 
\left\{\log(\hat{q}^2) + g_k(\mathfrak{p})\right\}
+ O_{k, \mathfrak{p}}\big(q^{-(1-\eta)/2}(\log q)^4\big).
$$
\end{proposition}

\begin{proof}
Consid\'erons :
\[
J:=\frac{1}{2\pi {\rm i}}
\int_{(2)} \Lambda(s+{\textstyle\frac{1}{2}}, f)^2 G(s)^2 \frac{{\rm d}s}{s},
\]
o\`u $G$ est un polyn\^ome de degr\'e $\geq 2$ v\'erifiant (\ref{defG}).
Par le th\'eor\`eme des r\'esidus, l'\'equation fonctionnelle (\ref{EF}) et le fait que
\footnote{C'est cette relation qui permet d'\'eviter le recours \`a une forme explicite de $\varepsilon_f$ 
que l'on n'a pas au niveau $q$ avec des facteurs carr\'es. C'est aussi pour cette raison que l'on ne peut actuellement pas avoir l'ordre exact du premier moment $M_{1}$ mais au mieux une majoration.} 
$\varepsilon_f^{2}=1$, 
on a 
\[
2J=\mathop{\rm Res}_{s=0}\left(\Lambda(s+{\textstyle\frac{1}{2}}, f)^2 \frac{G(s)^2}{s}\right)
= \hat{q} \Gamma(k/2)^2  L({\textstyle\frac{1}{2}}, f)^2.
\]
D'autre part,
la formule (\ref{5}) nous permet d'\'ecrire, avec la notation (\ref{zetaq}),
\begin{align*}
L(s+{\textstyle\frac{1}{2}}, f)^2
& = \sum_{a, b\ge 1} \frac{\lambda_f(a)\lambda_f(b)}{(ab)^{s+1/2}}
\\
& = \sum_{a, b\ge 1} \frac{1}{(ab)^{s+1/2}} 
\sum_{\substack{d\mid (a, b)\\ (d, q)=1}} \lambda_f\bigg(\frac{ab}{d^2}\bigg)
\\
& = \zeta^{(q)}(1+2s) \sum_{n\geq 1} \frac{\tau(n)\lambda_f(n)}{n^{s+1/2}}
\qquad
(\re s>\textstyle\frac{1}{2}).
\end{align*}
Ceci implique que
$$
2J = \hat{q} \sum_{n\geq 1} \frac{\tau(n)\lambda_f(n)}{\sqrt{n}} 
\frac{1}{\pi {\rm i}}
\int_{(2)} \Gamma(s+k/2)^2
\frac{G(s)^2}{s} \bigg(\frac{n}{\hat{q}^2}\bigg)^{-s} {\rm d}s.
$$
Les deux \'egalit\'es pr\'ec\'edentes donnent donc :
\begin{equation}\label{71}
L({\textstyle\frac{1}{2}}, f)^2 
= \sum_{n\geq1} \frac{\tau(n)}{\sqrt{n}} U\bigg(\frac{n}{\hat{q}^2}\bigg) \lambda_f(n)
\end{equation}
o\`u $U(y)$ est d\'efinie en (\ref{defU}).
En reportant cette expression dans (\ref{defXm}) et en utilisant le Corollaire \ref{cor_3}, il suit
\begin{equation}\label{M2m}
M_{2, m}
= \frac{\varphi(q)}{q} \frac{\tau(m)}{\sqrt{m}} U\bigg(\frac{m}{\hat{q}^2}\bigg) 
+ O_{k, \mathfrak{p}}({\mathscr R}_2),
\end{equation}
o\`u 
$$
{\mathscr R}_2
:= \sum_{n\geq1} \frac{\tau(n)}{\sqrt{n}} \bigg|U\bigg(\frac{n}{\hat{q}^2}\bigg)\bigg|
\left(\frac{\sqrt{mn} \{\log(2(m,n))\}^2}{q^{3/2}}
+\frac{\tau(m)\tau(n)}{q}\right).
$$
\`A l'aide de (\ref{PU}), il est facile de majorer la contribution du premier membre dans la parenth\`ese :
\begin{align*}
& \ll \frac{\sqrt{m} (\log q)^3}{q^{3/2}}
\sum_{n\leq q} \tau(n)
+ \sqrt{mq} (\log q)^2
\sum_{n> q} \frac{\tau(n)}{n^2} 
\\
& \ll \frac{\sqrt{m} (\log q)^4}{q^{1/2}}.
\end{align*}
De m\^eme  la contribution de $\tau(m)\tau(n)/q$ est $\ll \tau(m)(\log q)^4/\sqrt{q}$.
Ces deux estimations impliquent que
$$
{\mathscr R}_2
\ll_{k, \mathfrak{p}} q^{-(1-\eta)/2}(\log q)^4.
$$
En reportant dans (\ref{M2m}) et en utilisant la premi\`ere relation de (\ref{PU}),
on obtient le r\'esultat souhait\'e.
\end{proof}

\section{Calcul du troisi\`eme moment}

L'objectif de ce paragraphe est de d\'emontrer le r\'esultat suivant.

\begin{proposition}\label{pro_16}
Soient $k\ge 2$ un entier pair, $\mathfrak{p}$ un nombre premier
et $q=\mathfrak{p}^{\nu}$ avec $\nu \geq 3$.
On a 
$$
M_3
= 4\bigg(\frac{\varphi(q)}{q}\bigg)^4 \bigg\{
\frac{1}{3} (\log\hat{q})^3
+ \bigg(2\frac{\log\mathfrak{p}}{\mathfrak{p}-1}  + \frac{\Gamma'}{\Gamma}(k/2) + 2\gamma\bigg)
(\log\hat{q})^2+O_{k, \mathfrak{p}}(\log q) 
\bigg\},
$$
o\`u la constante implqu\'ee ne d\'epend que de $k$ et $\mathfrak{p}$.
\end{proposition}

\subsection{D\'ebut de la d\'emonstration de la Proposition \ref{pro_16}} 

\begin{lemme}\label{ExpM3}
Soient $k\ge 2$ un entier pair 
et $m, n, q\ge 1$ des entiers positifs.
Alors 
\begin{equation}\label{98}
M_3
=2\sum_{m,n\geq{1}}
\frac{\tau(m)}{\sqrt{mn}}
U\bigg(\frac{m}{\hat{q}^2}\bigg) T\bigg(\frac{n}{\hat{q}}\bigg) 
\Delta^*_q(m, n).
\end{equation}
\end{lemme}

\begin{proof}
On consid\`ere maintenant l'int\'egrale 
\[
I= \frac{1}{2{\rm i}\pi} \int_{(2)} \Lambda(s+{\textstyle\frac{1}{2}}, f)\frac{G(s)}{s} {\rm d}s.
\]
A l'aide de l'\'equation fonctionnelle (\ref{EF}), 
le th\'eor\`eme des r\'esidus nous permet d'\'ecrire 
\[
(1+\varepsilon_f) I
= \mathop{\rm Res}_{s=0}\bigg(\Lambda(s+{\textstyle\frac{1}{2}}, f) \frac{G(s)}{s}\bigg)
= \sqrt{\hat{q}} L({\textstyle\frac{1}{2}}, f)\Gamma(k/2).
\]
Cette \'egalit\'e et la s\'erie de Dirichlet (\ref{3}) donnent alors
\begin{equation}\label{99}
L({\textstyle\frac{1}{2}}, f) 
= (1+\varepsilon_f) \sum_{n\geq1} \frac{\lambda_f(n)}{\sqrt{n}} T\bigg(\frac{n}{\hat{q}}\bigg).
\end{equation}
Les \'egalit\'es
\footnote{C'est la diff\'erence entre les \'egalit\'es (\ref{99}) et (\ref{71}) 
qui emp\^eche de d\'eterminer l'ordre exact du premier moment des fonctions $L$-automorphes 
et qui nous conduit \`a \'etudier plut\^ot $M_{2}$ et $M_{3}$.}(\ref{99}) 
et (\ref{71}) impliquent que
\[
L({\textstyle\frac{1}{2}}, f)^3
=(1+\varepsilon_f) 
\sum_{m\geq1}\tau(m) \frac{\lambda_f(m)}{\sqrt{m}} U\bigg(\frac{m}{\hat{q}^2}\bigg)
\sum_{n\geq{1}} \frac{\lambda_f(n)}{\sqrt{n}} T\bigg(\frac{n}{\hat{q}}\bigg).
\]
Si $\varepsilon_f=1$ alors
$$
L({\textstyle\frac{1}{2}}, f)^3
=2\sum_{m\geq1}\tau(m) \frac{\lambda_f(m)}{\sqrt{m}} U\bigg(\frac{m}{\hat{q}^2}\bigg)
\sum_{n\geq{1}} \frac{\lambda_f(n)}{\sqrt{n}} T\bigg(\frac{n}{\hat{q}}\bigg).
$$
Mais ceci reste \'egalement vrai si $\varepsilon_f=-1$ : 
dans ce cas le membre de gauche est nul en vertu de l'\'equation fonctionnelle (\ref{EF}) qui impose alors $L({\textstyle\frac{1}{2}}, f)=0$ ; le membre de droite aussi est nul, en effet, 
de $L(\textstyle\frac{1}{2}, f)=0$ on d\'eduit $L(\textstyle\frac{1}{2}, f)^2 =0$ et donc 
\[
\sum_{m\geq1} \tau(m) \frac{\lambda_f(m)}{\sqrt{m}} U\bigg(\frac{m} {\hat{q}^2}\bigg)=0
\]
gr\^ace \`a (\ref{71}).

Finalement, on a pour toute forme primitive de niveau $q$ :
\begin{equation}\label{101}
L({\textstyle\frac{1}{2}}, f)^3
=2\sum_{m,n\geq1} \frac{\tau(m)}{\sqrt{mn}}
U\bigg(\frac{m}{\hat{q}^2}\bigg) T\bigg(\frac{n}{\hat{q}}\bigg) 
\lambda_f(m)\lambda_f(n).
\end{equation}
Ce qui implique le r\'esultat d\'esir\'e.
\end{proof}

\subsection{Application de la formule de trace} 
En appliquant la formule de trace du Corollaire \ref{cor_3} \`a l'\'egalit\'e (\ref{98}), on peut \'ecrire
\begin{equation}\label{102}
M_3
= 2 \frac{\varphi(q)}{q} \suma_{n\geq{1}} \frac{\tau(n)}{n} 
T\bigg(\frac{n}{\hat{q}}\bigg) U\bigg(\frac{n}{\hat{q}^2}\bigg)
+ O_{k, \mathfrak{p}}({\mathscr R}_3 + {\mathscr R}_4)
\end{equation}
avec
\begin{align*}
{\mathscr R}_3
& := \sum_{m,n\geq{1}} \frac{\tau(m)\{\log2(m,n)\}^2}{q^{3/2}}
\bigg|T\bigg(\frac{n}{\hat{q}}\bigg) U\bigg(\frac{m}{\hat{q}^2}\bigg)\bigg|,
\\
{\mathscr R}_4
& := \sum_{m,n\geq 1} \frac{\tau(m)^2\tau(n)}{q\sqrt{mn}}
\bigg|T\bigg(\frac{n}{\hat{q}}\bigg) U\bigg(\frac{m}{\hat{q}^2}\bigg)\bigg|,
\end{align*}
o\`u $\sum^*_{n\ge 1}$ d\'esigne la somme portant sur les entiers $n$ tels que $(n, q)=1$.

\subsection{\'Evaluation du terme principal} 
Afin de calculer le premier terme de droite de (\ref{102}), \'ecrivons 
\begin{equation}\label{107}
\begin{aligned}
\suma_{n\geq 1} \frac{\tau(n)}{n} T\bigg(\frac{n}{\hat{q}}\bigg) U\bigg(\frac{n}{\hat{q}^2}\bigg)
& = \bigg(\suma_{n\leq q} + \suma_{n>q}\bigg)
\frac{\tau(n)}{n} T\bigg(\frac{n}{\hat{q}}\bigg) U\bigg(\frac{n}{\hat{q}^2}\bigg).
\end{aligned}
\end{equation}

En faisant appel \`a (\ref{PT})-(\ref{PU}) avec $j=1$ et (\ref{estimation4}), on a
\begin{equation}\label{108}
\suma_{n>q} \frac{\tau(n)}{n} T\bigg(\frac{n}{\hat{q}}\bigg) U\bigg(\frac{n}{\hat{q}^2}\bigg)
\ll q^{3/2} \sum_{n>q} \frac{\tau(n)}{n^3}
\ll q^{-1/2}\log q.
\end{equation}

En utilisant la premi\`ere relation de (\ref{PU}), on peut \'ecrire
\begin{equation}\label{termeprincipal}
\suma_{n\leq q} \frac{\tau(n)}{n} T\bigg(\frac{n}{\hat{q}}\bigg)
U\bigg(\frac{n}{\hat{q}^2}\bigg)
= {\mathscr T} + O({\mathscr R}_5),
\end{equation}
o\`u 
$$
{\mathscr T}
:= \frac{\varphi(q)}{q}
\suma_{n\leq q} \frac{\tau(n)}{n} T\bigg(\frac{n}{\hat{q}}\bigg)
\bigg\{\log\bigg(\frac{\hat{q}^2}{n}\bigg) + g_k(\mathfrak{p})\bigg\}
$$
et
\begin{equation}\label{MajR5}
\begin{aligned}
{\mathscr R}_5
& := \frac{1}{\hat{q}^2}\suma_{n\leq q} \tau(n) \left|T\bigg(\frac{n}{\hat{q}}\bigg)\right|
\\
& \ll \frac{1}{\hat{q}^2}
\bigg(
\sum_{n\leq \hat{q}} \tau(n) 
+ \hat{q}^2\suma_{\hat{q}<n\leq q} \frac{\tau(n)}{n^2}
\bigg)
\\
& \ll q^{-1/2}\log q
\end{aligned}
\end{equation}
gr\^ace \`a (\ref{PT}), (\ref{estimation2}) et (\ref{estimation4}).

Pour \'evaluer le terme principal ${\mathscr T}$, on \'ecrit, \`a l'aide de (\ref{defT}),
$$
{\mathscr T} = \frac{\varphi(q)}{q} \big({\mathscr T}_0 - {\mathscr R}_6\big),
$$
o\`u
\begin{align*}
{\mathscr T}_0
& := \frac{1}{2\pi {\rm i}} \int_{(2)} 
\frac{\Gamma\left(s+\frac{k}{2}\right)}{\Gamma(\frac{k}{2})}
\suma_{n\geq 1}
\frac{\tau(n)}{n^{s+1}}
\bigg\{\log\bigg(\frac{\hat{q}^2}{n}\bigg) + g_k(\mathfrak{p})\bigg\}
\hat{q}^s \frac{G(s)}{s} {\rm d}s,
\\
{\mathscr R}_6
& := \frac{1}{2\pi {\rm i}} \int_{(2)} 
\frac{\Gamma\left(s+\frac{k}{2}\right)}{\Gamma(\frac{k}{2})}
\suma_{n>q}
\frac{\tau(n)}{n^{s+1}}
\bigg\{\log\bigg(\frac{\hat{q}^2}{n}\bigg) + g_k(\mathfrak{p})\bigg\}
\hat{q}^s \frac{G(s)}{s} {\rm d}s.
\end{align*}

En utilisant l'estimation (\ref{estimation4}) du Lemme \ref{lem_estimation} pour majorer la somme 
dans ${\mathscr R}_6$ et la formule de Stirling
\begin{equation}\label{Stirling}
|\Gamma(s)|
= \sqrt{2\pi} \, e^{-(\pi/2)|\tau|} |\tau|^{\sigma-1/2}
\big\{1 + O_\sigma\big(|\tau|^{-1}\big)\big\}
\end{equation}
valable uniform\'ement pour $|\tau|\ge 1$,
on peut d\'eduire que
\begin{equation}\label{MajR6}
{\mathscr R}_6
\ll_k q^{-1}(\log q)^2.
\end{equation}

Pour \'evaluer ${\mathscr T}_0$,
on \'ecrit d'abord
\begin{align*}
{\mathscr T}_0
& = \big(\log\hat{q}^2+g_k(\mathfrak{p})\big)
\frac{1}{2\pi {\rm i}}
\int_{(2)} \frac{\Gamma(s+\frac{k}{2})}{\Gamma(\frac{k}{2})}
\zeta^{(q)}(s+1)^2 \hat{q}^s \frac{G(s)}{s} {\rm d}s
\\
& \quad
+ \frac{1}{\pi {\rm i}}
\int_{(2)} \frac{\Gamma(s+\frac{k}{2})}{\Gamma(\frac{k}{2})} 
\zeta^{(q)}(s+1)\zeta^{(q)\prime}(s+1) \hat{q}^s \frac{G(s)}{s} {\rm d}s.
\end{align*}
Ensuite
on utilise le th\'eor\`eme des r\'esidus autour du p\^ole $s=0$, 
les int\'egrales r\'esultantes en $\sigma=-\frac{1}{2}$ sont en $O_{k, \mathfrak{p}}(q^{-1/4}\log q)$ 
de sorte que 
\begin{equation}\label{T0}
\begin{aligned}
{\mathscr T}_0
& = \big(2\log\hat{q}+g_k(\mathfrak{p})\big)
\mathop{\rm Res}_{s=0} \bigg(\zeta^{(q)}(s+1)^{2} \frac{F(s)}{s}\bigg)
\\
& \quad
+ 2 \mathop{\rm Res}_{s=0} \bigg(\zeta^{(q)}(s+1)\zeta^{(q)\prime}(s+1) \frac{F(s)}{s}\bigg)
+ O_{k, \mathfrak{p}}(q^{-1/4}\log q)
\end{aligned}
\end{equation}
o\`u
\[
F(s) := \frac{\Gamma(s+\frac{k}{2})}{\Gamma(\frac{k}{2})} \hat{q}^s G(s).
\]
Un calcul \'el\'ementaire montre que 
\begin{equation}\label{Fj}
F^{(j)}(0)
= \sum_{0\le i\le j} \xi_{j, i} (\log\hat{q})^{j-i},
\end{equation}
o\`u
\begin{align*}
\xi_{j, 0} 
& = 1
\quad(0\le j\le 3),
\\
\xi_{j, 1} 
& :=  j\frac{\Gamma'}{\Gamma}(k/2)
\quad(1\le j\le 3),
\\
\xi_{j, 2} 
& :=  (2j-3) \bigg(\frac{\Gamma''}{\Gamma}(k/2)+G''(0)\bigg)
\quad(2\le j\le 3),
\\
\xi_{3, 3} 
& :=  \frac{\Gamma'''}{\Gamma}(k/2)+3\frac{\Gamma'}{\Gamma}(k/2)G''(0),
\end{align*}

En utilisant les relations (\ref{zeta1s}) et (\ref{p1s}), on trouve
\begin{align}
\zeta^{(q)}(s+1)^{2}
& = \bigg(\frac{\varphi(q)}{q}\bigg)^2 
\bigg\{\frac{a_{-2}}{s^2} + \frac{a_{-1}}{s} + a_0 + O(s)\bigg\},
\label{zetazeta}
\\
\zeta^{(q)}(s+1)\zeta^{(q)\prime}(s+1)
& = \bigg(\frac{\varphi(q)}{q}\bigg)^2 
\bigg\{\frac{b_{-3}}{s^3} + \frac{b_{-2}}{s^2} + b_0 + O(s)\bigg\},
\label{zetazeta'}
\end{align}
o\`u
\begin{align*}
a_{-2}
& :=1,
\\
a_{-1}
& := 2\frac{\log\mathfrak{p}}{\mathfrak{p}-1}+2\gamma_0,
\\
a_0
& := \bigg(\frac{\log\mathfrak{p}}{\mathfrak{p}-1}\bigg)^2
 - \frac{(\log\mathfrak{p})^2-4\gamma_0\log\mathfrak{p}}{\mathfrak{p}-1}
+ \gamma_0^2 - 2\gamma_1,
\\
b_{-3}
& := -1,
\\
b_{-2}
& := - \frac{\log\mathfrak{p}}{\mathfrak{p}-1} - \gamma_0,
\\
b_0
& := - \frac{(\log\mathfrak{p})^3-2\gamma_0(\log\mathfrak{p})^2}{2(\mathfrak{p}-1)^2}
+ \frac{(\log\mathfrak{p})^3-6\gamma_0(\log\mathfrak{p})^2+6(\gamma_0^2-2\gamma_1)\log\mathfrak{p}}{6(\mathfrak{p}-1)}
\\
& \quad
- \gamma_0\gamma_1 + \frac{\gamma_2}{2}.
\end{align*}
Donc on a les r\'esidus suivants :
\begin{align*}
\mathop{\rm Res}_{s=0} \bigg(\zeta^{(q)}(s+1)^2 \frac{F(s)}{s}\bigg)
& = \bigg(\frac{\varphi(q)}{q}\bigg)^2
\bigg(\frac{a_{-2}}{2}F''(0) + a_{-1}F'(0) + a_0F(0)\bigg),
\\
\mathop{\rm Res}_{s=0}  \bigg(\zeta^{(q)}(s+1) \zeta^{(q)\prime}(s+1) \frac{F(s)}{s}\bigg)
& = \bigg(\frac{\varphi(q)}{q}\bigg)^2
\bigg(\frac{b_{-3}}{6}F'''(0) + \frac{b_{-2}}{2}F''(0) + b_0F(0)\bigg).
\end{align*}
En reportant dans (\ref{T0}) et en utilisant (\ref{Fj}),
on obtient
\begin{equation}\label{T0a}
{\mathscr T}_0
= \bigg(\frac{\varphi(q)}{q}\bigg)^2
Q(\log\hat{q}) + O_{k, \mathfrak{p}}(q^{-1/4}\log q),
\end{equation}
o\`u
\begin{equation}\label{109}
Q(X) := A_3 X^3 + A_2 X^2 + A_1 X + A_0,
\end{equation}
et les constantes $A_j=A_j(k, \mathfrak{p})$ sont donn\'ees par
\begin{align*}
A_3
& :=  \textstyle
a_{-2}\xi_{2, 0} + \frac{1}{3}b_{-3}\xi_{3, 0},
\\\noalign{\vskip 1mm}
A_2
& := \textstyle
a_{-2}\xi_{2, 1} + 2a_{-1}\xi_{1, 0} + \frac{1}{2}a_{-2}\xi_{2, 0}g_k(\mathfrak{p})
+ \frac{1}{3}b_{-3}\xi_{3, 1} + b_{-2}\xi_{2, 0},
\\\noalign{\vskip 1mm}
A_1
& := \textstyle
a_{-2}\xi_{2, 2} + 2a_{-1}\xi_{1, 1} + 2a_0\xi_{0, 0} 
+ (\frac{1}{2}a_{-2}\xi_{2, 1} + a_{-1}\xi_{1, 0})g_k(\mathfrak{p})
+ \frac{1}{3}b_{-3}\xi_{3, 2} + b_{-2}\xi_{2, 1},
\\\noalign{\vskip 1mm}
A_0
& := \textstyle
(\frac{1}{2}a_{-2}\xi_{2, 2} + a_{-1}\xi_{1, 1} + a_0\xi_{0, 0})g_k(\mathfrak{p})
+ \frac{1}{3}b_{-3}\xi_{3, 3} + b_{-2}\xi_{2, 2} + 2b_0\xi_{0, 0}.
\end{align*}
En combinant (\ref{T0a}), (\ref{MajR5}), (\ref{MajR6}), (\ref{termeprincipal}), (\ref{108}) avec (\ref{107}), 
on trouve
\begin{equation}\label{110}
\suma_{n\geq 1} \frac{\tau(n)}{n} T\bigg(\frac{n}{\hat{q}}\bigg) U\bigg(\frac{n}{\hat{q}^2}\bigg)
= \bigg(\frac{\varphi(q)}{q}\bigg)^3 Q(\log\hat{q})+O_{k, \mathfrak{p}}(q^{-1/4}\log q).
\end{equation}

\subsection{Estimation pour le terme d'erreur ${\mathscr R}_3$} 
\'Ecrivons
\begin{align*}
{\mathscr R}_3
& = \frac{1}{q^{3/2}}
\sum_{a\geq 1} \{\log(2a)\}^2 
\sum_{\substack{m, n\geq 1\\ (m, n)=a}} \tau(m)
\bigg|T\bigg(\frac{n}{\hat{q}}\bigg) U\bigg(\frac{m}{\hat{q}^2}\bigg)\bigg|
\\
& = \frac{1}{q^{3/2}}
\sum_{a\geq 1} \{\log(2a)\}^2 
\sum_{\substack{m, n\geq 1\\ (m, n)=1}} \tau(am)
\bigg|T\bigg(\frac{an}{\hat{q}}\bigg) U\bigg(\frac{am}{\hat{q}^2}\bigg)\bigg|
\\
& \ll \frac{1}{q^{3/2}}
\sum_{a\geq 1} \{\log(2a)\}^2 \sum_{b\ge 1} |\mu(b)|
\sum_{m, n\geq 1} \tau(abm)
\bigg|T\bigg(\frac{abn}{\hat{q}}\bigg) U\bigg(\frac{abm}{\hat{q}^2}\bigg)\bigg|
\\
& \ll \frac{1}{q^{3/2}}
\sum_{d\ge 1} h(d)\tau(d)
\sum_{n\geq 1} 
\bigg|T\bigg(\frac{dn}{\hat{q}}\bigg)\bigg|
\sum_{m\geq 1} \tau(m)
\bigg|U\bigg(\frac{dm}{\hat{q}^2}\bigg)\bigg|,
\end{align*}
o\`u
$$
h(d):=\sum_{ab=d} \{\log(2a)\}^2|\mu(b)|.
$$

En utilisant (\ref{PT}), on a :
$$
\sum_{n\geq 1} 
\bigg|T\bigg(\frac{dn}{\hat{q}}\bigg)\bigg|
\ll \sum_{n\le \hat{q}/d} 1 
+ \sum_{n>\max\{\hat{q}/d, 1\}} \bigg(\frac{\hat{q}}{dn}\bigg)^2
\ll \frac{\hat{q}}{d}.
$$
De fa{\c c}on similaire, les estimations (\ref{PU}) avec $j=2$,
(\ref{estimation2}) et (\ref{estimation4}) nous permettent de d\'eduire
\begin{align*}
\sum_{m\geq 1} \tau(m)
\bigg|U\bigg(\frac{dm}{\hat{q}^2}\bigg)\bigg|
& \ll \sum_{m\le \hat{q}^2/d} \tau(m)\log\bigg(\frac{\hat{q}^2}{dm}\bigg)
+ \sum_{m>\max\{\hat{q}^2/d, 1\}} \tau(m)\bigg(\frac{\hat{q}^2}{dm}\bigg)^2
\\
& \ll \frac{\hat{q}^2\log q}{d}.
\end{align*}
En combinant ces estimations, on obtient
\begin{equation}\label{MajR1}
{\mathscr R}_3
\ll (\log q)
\sum_{d\ge 1} \frac{h(d)\tau(d)}{d^2}
\ll \log q.
\end{equation}

\subsection{Estimation pour le terme d'erreur ${\mathscr R}_4$} 
Appliquant (\ref{PT}) avec $j=2$ et (\ref{estimation3})-(\ref{estimation4}), on a
\begin{equation}\label{104}
\begin{aligned}
\sum_{n\geq 1} \frac{\tau(n)}{\sqrt{n}} \bigg|T\bigg(\frac{n}{\hat{q}}\bigg)\bigg|
& \ll \sum_{n\leq \hat{q}} \frac{\tau(n)}{\sqrt{n}}
+ \hat{q}^2 \sum_{n>\hat{q}} \frac{\tau(n)}{n^{5/2}}
\\
& \ll q^{1/4}\log q.
\end{aligned}
\end{equation}
De m\^eme (\ref{PU}) avec $j=2$, (\ref{estimation3}) et (\ref{estimation4}) impliquent
\begin{equation}\label{105}
\begin{aligned}
\sum_{m\geq 1} \frac{\tau(m)^2}{\sqrt{m}} \bigg|U\bigg(\frac{m}{\hat{q}^2}\bigg)\bigg|
& \ll \log q 
\sum_{m\leq \hat{q}^2} \frac{\tau(m)^2}{\sqrt{m}}
+ \hat{q}^4 \sum_{m>\hat{q}^2} \frac{\tau(m)^3}{m^{5/2}}
\\
& \ll q^{1/2}(\log q)^7.
\end{aligned}
\end{equation}
En combinant (\ref{104}) et (\ref{105}), on obtient :
\begin{equation}\label{MajR2}
{\mathscr R}_4
\ll q^{-1/4}(\log q)^8.
\end{equation}

\subsection{Fin de la d\'emonstration de la Proposition~\ref{pro_16}}
En reportant (\ref{110}), (\ref{MajR2}) et (\ref{MajR1}) dans (\ref{102}),
on obtient 
$$
M_3
= 2\bigg(\frac{\varphi(q)}{q}\bigg)^4 Q(\log\hat{q}) + O_{k, \mathfrak{p}}(\log q).
$$
Un calcul \'el\'ementaire montre que
$$
A_3=\frac{2}{3},
\qquad
A_2
= 2\bigg(2\frac{\log\mathfrak{p}}{\mathfrak{p}-1}  + \frac{\Gamma'}{\Gamma}(k/2) + 2\gamma\bigg).
$$
Ceci implique le r\'esultat annonc\'e.

\section{D\'emonstration du Th\'eor\`eme~\ref{thm_1}}

On utilise l'in\'egalit\'e de H\"older, selon laquelle :
\[
\bigg(\sumh_{f\in {\rm H}^{*}_{k}(q)} L({\textstyle\frac{1}{2}}, f)^2 \bigg)^3
\leq \bigg(\sumh_{f\in{ \rm H}^{*}_{k}(q)} L({\textstyle\frac{1}{2}}, f)^3\bigg)^2
\sumh_{\substack{f\in {\rm H}^{*}_{k}(q)\\ L(\frac{1}{2}, f)\neq{0}}} 1.
\]
On en d\'eduit :
\[
\sumh_{\substack{f\in {\rm H}^{*}_{k}(q)\\ L(\frac{1}{2}, f)\neq{0}}}1
\geq \frac{M_2^3}{M_3^2}.
\]
Utilisons les Propositions \ref{pro_12}-\ref{pro_16} pour obtenir avec l'in\'egalit\'e pr\'ec\'edente :
\[
\sumh_{\substack{f\in {\rm H}^{*}_{k}(q)\\ L(\frac{1}{2}, f)\neq{0}}}1
\gg_{k,\mathfrak{p}} \frac{\left((\varphi(q)/q)^2\log{q}+O(1)\right)^3}{(\log{q})^6}
\gg_{k, \mathfrak{p}} \frac{1}{(\log{q})^3}
\]
ce qui termine la preuve du Th\'eor\`eme~\ref{thm_1}.

\end{document}